\documentclass{article}%
\usepackage{amssymb}
\usepackage{amsmath}
\usepackage{amsfonts}
\usepackage{graphicx}%
\newtheorem{theorem}{Theorem}[section]

\newtheorem{proposition}{Proposition}[section]
\newtheorem{notation}{Notation}[section]
\newtheorem{lemma}{Lemma}[section]

\newtheorem{definition}{Definition}[section]

\newtheorem{remark}{Remark}[section]
\newenvironment{proof}[1][Proof]{\noindent\textbf{#1.} }{\ \rule{0.5em}{0.5em}}
\begin{document}

\title{Risk-sensitive Necessary and Sufficient Optimality Conditions and Financial
Applications: \\Fully Coupled Forward-Backward Stochastic Differential Equations with Jump
diffusion }
\author{\textbf{Rania KHALLOUT}$^{a\dagger}$\textbf{\ and Adel CHALA}$^{a\ddagger}%
$\textbf{\ }\\a) Laboratory of Applied Mathematics.\\Mohamed Khider University. \\P.O. Box 145, Biskra 07000. Algeria\\E-mail: $\ddagger$adel.chala@univ-biskra.dz.\\E-mail: $\dagger$rania\_khallout@yahoo.com.}
\maketitle

\begin{abstract}
Throughout this paper, we focused our aim on the problem of optimal control
under a risk-sensitive performance functional, where the system is given by a
fully coupled forward-backward stochastic differential equation with jump. The
risk neutral control system has been used as preliminary step, where the
admissible controls are convex, and the optimal solution exists. The necessary
as well as sufficient optimality conditions for risk-sensitive performance are
proved. At the end of this work, we illustrate our main result by giving an
example of mean-variance for risk sensitive control problem applied in cash
flow market.


\end{abstract}

\section{Introduction}

Maximum principle for controlled stochastic differential equations (SDE in
short), whose objective is to obtain necessary as well as sufficient
optimality conditions of controls, has been extensively investigated since
1970s. The initial work was done by Kushner \cite{Kushner}. The other
fundamental advance was developed by Haussmann \cite{Hausmann1,Hausmann2}.
Versions of the stochastic maximum principle ( SMP in short), in which the
diffusion coefficient is allowed to depend explicitly on the control variable,
have been derived by Arkin \& Saksonov \cite{Arkin}, Bensoussan
\cite{Bensoussan}, and Bismut \cite{Bismut1,Bismut2,Bismut3}. The results of
\cite{Arkin} and \cite{Bismut1,Bismut2,Bismut3} consider the case of random
coefficients. Necessary and sufficient optimality conditions for linear
systems with random coefficients, where no $L^{p}$-bounds are imposed on the
controls, are established by Cadellinas and Karatzas \cite{Cadenillas1}. The
general case, where the control domain is not convex, and the diffusion
coefficient depends explicitly on the variable control, was derived by Peng
\cite{Peng} by introducing two adjoint processes, and a variational inequality
of the second order. Recently, by considering risk sensitive performance
control with an exponential functional cost, Djehiche et al \cite{BTT}
generalized the previous results on the subject, and derive necessary
optimality conditions, by adding the mean field process.

The initial works on optimal control of jump processes was first considered by
Boel \cite{Boel, Boel varaiya}, Rishel \cite{Rishel}. Later, many authors
studied this kind of control problems including Situ \cite{Situ}, Cadellinas
\cite{Cadenillas2}, and Framstad \O ksendal \& Sulem \cite{FraOksendalSul}. We
note that in \cite{Cadenillas2} and \cite{FraOksendalSul}, some applications
in finance are treated. The general case, where the control domain is not
convex and the diffusion coefficient depends explicitly on the control
variable, was derived by Tang and Li \cite{TangLi}, by using the second order
expansion, the results of \cite{TangLi} are given with two adjoint processes
and a variational inequality of the second order. For more details on the
controlled systems with jumps and their applications, see \O ksendal and Sulem
\cite{Oksendal Sul2} and the references therein.

The purpose of this paper is to generalize the model governed with SDE and
BSDE, before that we must give this motivation example which has taken from
the thesis of Armerin \cite{Arm}.

Modeling and controlling cash flow processes of a firm or a project, such as
pricing and managing an insurance contract, is a class of problems where
forward backward stochastic differential equations (FBSDEs in short) provide a
natural setup and a powerful tool. In this paper, we shall investigate an
example of such a situation arising in the pricing of a simple insurance contract.

A policyholder at an insurance company has paid premiums that at time zero
have accumulated to the sum $m_{0}$. The money is invested in an asset
portfolio with wealth $\left(  x_{t}\right)  _{t\in\left[  0,T\right]  }$
managed by the insurance company under a time interval $\left[  0,T\right]  $.
At each instant $t\in\left[  0,T\right]  $, the policyholder ought to receive
an amount $c_{t}x_{t}$. The present value (price) of the cash stream $\left(
c_{s}x_{s}\right)  $, discounted to time $t$ with a discount factor (deflator)
$\exp\left\{  -%
{\displaystyle\int_{0}^{t}}
\lambda_{s}ds\right\}  $, where $\lambda_{t}$ is assumed nonnegative, bounded,
and deterministic, is given by%
\begin{equation}
y_{t}=\mathbb{E}\left[
{\displaystyle\int_{t}^{T}}
e^{-%
{\displaystyle\int_{0}^{s}}
\lambda_{r}dr}c_{s}\lambda_{s}ds\text{ }\left\vert \text{ }\mathcal{F}%
_{t}\right.  \right]  . \label{y conditional}%
\end{equation}

Assume that the portfolio is invested in a simple Black-Scholes market model
consisting of a risk-free asset (for example, a bond or a bank account) with a
short interest rate $r_{t}$ assumed bounded and deterministic, and a risky
asset evolving as a geometric Brownian motion with rate of return $\mu_{t}$
and volatility $\sigma_{t}$, both assumed to be bounded and deterministic
functions of time, with $\sigma_{t}\geq\varepsilon>0$ for all $t\in\left[
0,T\right]  $. In this market the wealth process $\left(  x_{t}\right)
_{t\in\left[  0,T\right]  }$ is governed by the dynamics given by%
\begin{equation}
\left\{
\begin{array}
[c]{l}%
dx_{t}=\left(  r_{t}x_{t}+\rho_{t}u_{t}\right)  dt+\sigma_{t}u_{t}dW_{t},\\
x_{0}=m_{0},
\end{array}
\right.  \label{x forward}%
\end{equation}

where $u_{t}$ is the amount invested in the risky asset and $\rho_{t}=\mu
_{t}-r_{t}$ is the risk premium held for this investment.

The insurance company allocates the amounts $\left(  u_{t}\right)  $ in order
to come close to the following target at time $T$: Find the admissible
strategies $\left(  c,u\right)  $ which maximize the policyholder's
preferences represented by the utility function $F$ of the cash streams, under
the condition that the total amount to be paid out is equal to the total
premium $m_{0}$:%
\begin{equation}
\max_{\left(  c,u\right)  }\frac{1}{\theta}\mathbb{E}\left[  F^{\theta}\left(
x_{t}\right)  \right]  . \label{cost1}%
\end{equation}

By selecting an appropriate portfolio choice strategy $u\left(  .\right)  ,$
where the exponent $\theta>0$ is called the risk sensitive parameter. Assume
that the policyholder's utility function is of HARA (hyperbolic absolute risk
aversion) type. That is, $F\left(  X\right)  =\frac{X^{\theta}}{\theta}$,
where $\theta\in\left(  0,1\right)  $. We can rewrite the expectation $\left(
\ref{cost1}\right)  $ $\mathbb{E}\left[  F^{\theta}\left(  x_{t}\right)
\right]  $ in terms of an expected exponential of integral criterion, by
applying It\^{o}'s formula to $\ln x_{t}^{\theta}=\theta\ln x_{t},$ we get%
\[
\max_{\left(  c,u\right)  }m_{0}^{\theta}\mathbb{E}\left[
{\displaystyle\int_{0}^{T}}
\exp\theta\left\{  f\left(  t,x_{t},u_{t}\right)  dt\right\}  \right]  ,
\]

where \
\[
f\left(  t,x_{t},u_{t}\right)  =\frac{\left(  \theta-1\right)  \sigma^{2}}%
{2}u_{t}^{2}+\left(  \frac{1}{2}\sigma^{2}+m-r_{t}-c_{t}x_{t}\right)
u_{t}+r_{t},
\]

and%
\[
y_{0}=\mathbb{E}\left[
{\displaystyle\int_{0}^{T}}
e^{-%
{\displaystyle\int_{0}^{s}}
\lambda_{r}dr}c_{s}\lambda_{s}ds\text{ }\left\vert \text{ }\mathcal{F}%
_{0}\right.  \right]  ,
\]

is the total value of the stream of cash flows discounted to time zero.

We need the following definition of admissible strategies suitable for our problem.

\begin{definition}
An admissible strategy is a pair of $\left(  \mathcal{F}_{t}\right)  _{t\geq
0}$-adapted processes $\left(  c,u\right)  $ such that $\left(
\ref{x forward}\right)  $ has a strong solution $\left(  x_{t}\right)
_{t\in\left[  0,T\right]  }$ that satisfies
\[
\mathbb{E}\int_{0}^{T}\left\vert x_{t}\right\vert dt<\infty,
\]
and
\[
\mathbb{E}\left[  \int_{0}^{T}e^{-%
{\displaystyle\int_{0}^{t}}
\lambda_{s}ds}c_{t}\lambda_{t}dt\right]  ^{2}<\infty.
\]

\end{definition}

Now, for each admissible strategy $\left(  c,u\right)  $, the $\left(
\mathcal{F}_{t}\right)  _{t\geq0}$-adapted value process $\left(
y_{t}\right)  _{t\geq0}$ in $\left(  \ref{y conditional}\right)  $ satisfies
the following BSDE:%
\begin{equation}
\left\{
\begin{array}
[c]{l}%
dy_{t}=\left(  \lambda_{t}y_{t}-c_{t}x_{t}\right)  dt+z_{t}dW_{t},\\
y_{T}=0,
\end{array}
\right.  \label{y backward}%
\end{equation}

where $\left(  z_{t}\right)  _{t\geq0}$ is $\left(  \mathcal{F}_{t}\right)
_{t\geq0}$-adapted and square-integrable with respect to $dt\times
d\mathbb{P}$ over $\left[  0,T\right]  \times\Omega.$

Hence, $\left(  \ref{x forward}\right)  $ and $\left(  \ref{y backward}%
\right)  $ satisfied by $\left(  x,y,z\right)  $ is a FBSDE, in the next step
we want to improve this notion of cash flow problem into a system of fully
coupled FBSDE\ with jump diffusion, as the best of our acknowledge, this is
not a simple or trivial extension, because of we have a lot of work to do.
Firstly the function minimize has the form an expected exponential, secondly
the problem of control governed by a fully coupled FBSDE with jump diffusion
as in system $\left(  \ref{problem 1}\right)  $ is very hard to solve it
especially if we want to derive the stochastic maximum principle (Lemma
$\ref{Lemma p2 p3},$ $\ref{NOC},$ and $\ref{SOC}$ below$)$.

Our aim in this paper is to derive necessary as well as sufficient optimality
conditions for jump process, controlled diffusion and generator for the system
driven by a fully coupled forward backward stochastic differential equation
(FBSDE in short) under a risk sensitive performance. We give the results, in
the form of global SMP, by using an auxiliary process as a preliminary step
see the section 3 below.

In the risk sensitive performance case, the system is governed by a FBSDE with
jump diffusion%
\[
\left\{
\begin{array}
[c]{ll}%
dx^{v}\left(  t\right)  = & b\left(  t,x^{v}\left(  t\right)  ,y^{v}\left(
t\right)  ,z^{v}\left(  t\right)  ,r^{v}\left(  t,.\right)  ,v\left(
t\right)  \right)  dt\\
& +\sigma\left(  t,x^{v}\left(  t\right)  ,y^{v}\left(  t\right)
,z^{v}\left(  t\right)  ,r^{v}\left(  t,.\right)  ,v\left(  t\right)  \right)
dW\left(  t\right) \\
& +%
{\displaystyle\int_{\Gamma}}
\gamma\left(  t-,x\left(  t-\right)  ,y\left(  t-\right)  ,z\left(  t-\right)
,r\left(  t,.\right)  ,v\left(  t\right)  ,\lambda\right)  \tilde{N}\left(
dt,d\lambda\right) \\
dy^{v}\left(  t\right)  = & -g\left(  t,x^{v}\left(  t\right)  ,y^{v}\left(
t\right)  ,z^{v}\left(  t\right)  ,r^{v}\left(  t,.\right)  ,v\left(
t\right)  \right)  dt+z^{v}\left(  t\right)  dW\left(  t\right) \\
& +%
{\displaystyle\int_{\Gamma}}
r^{v}\left(  t,\lambda\right)  \tilde{N}\left(  dt,d\lambda\right) \\
x^{v}\left(  0\right)  = & d,\text{ \ \ \ }y\left(  T\right)
=a,\text{\ \ \ \ \ \ \ \ \ \ \ \ \ \ \ }t\in\left[  0,T\right]
\end{array}
\right.
\]
where $b,\ \sigma,$ $\gamma$ and $g$ are given functions, $d$ is the initial
data, $a$ is terminal data, and $W=\left(  W\left(  t\right)  \right)
_{t\geq0}$ is a standard Brownian motion defined on a filtered probability
space $\left(  \Omega,\mathcal{F},\left(  \mathcal{F}_{t}\right)  _{t\geq
0},\mathcal{P}\right)  $, satisfying the usual conditions, and $\widetilde
{N}\left(  dt,d\lambda\right)  $ is a Poisson martingale measure with
characteristic $m\left(  d\lambda\right)  dt.$

The control variable $v=\left(  v_{t}\right)  $, called strict control, is an
$\mathcal{F}_{t}$-adapted process with values in some set $U$ of $\mathbb{R}$.
We denote by $\mathcal{U}$ the class of all strict admissible controls. The
criteria to be minimized over $\mathcal{U}$ has the form%
\begin{align*}
&  J^{\theta}\left(  v\right) \\
&  =\mathbb{E}\left(  \exp\theta\left[
{\displaystyle\int_{0}^{T}}
f\left(  t,x^{v}\left(  t\right)  ,y^{v}\left(  t\right)  ,z^{v}\left(
t\right)  ,r^{v}\left(  t,.\right)  ,v\left(  t\right)  \right)
dt+\Phi\left(  x^{v}\left(  T\right)  \right)  +\Psi\left(  y^{v}\left(
0\right)  \right)  \right]  \right)  ,
\end{align*}
where $\Phi,$ $\Psi$ and $f$ are given maps and $\left(  x^{v}\left(
t\right)  ,y^{v}\left(  t\right)  \right)  $ is the trajectories controlled by
$v.$

A control $u\in\mathcal{U}$ is called optimal if it satisfies%
\[
J\left(  u\right)  =\underset{v\in\mathcal{U}}{\inf}J\left(  v\right)  .
\]

To achieve the objective of this paper, and establish the necessary and
sufficient optimality conditions, the existence and uniqueness of the optimal
control which minimize the functional cost is proved, we proceed as follows.
Firstly, we give the optimality conditions for risk neutral controls. The idea
is to use the fact that the auxiliary state process $\xi^{v}\left(  t\right)
$ is the best intermediate step to translate the system of forward backward
SDE into three equations see $\left(  \ref{problem 1}\right)  $ in section 3.
Secondly, we suggest a transformation of the adjoint equations $\left(
p_{1},q_{1}\right)  ,$ $\left(  p_{2},q_{2}\right)  ,$ $\left(  p_{3}%
,q_{3}\right)  $ and $\pi\left(  \lambda\right)  $ into following adjoint
equations $\left(  \widetilde{p}_{2},\widetilde{q}_{2}\right)  ,$ $\left(
\widetilde{p}_{3},\widetilde{q}_{3}\right)  $ and $\widetilde{\pi}\left(
\lambda\right)  $ by applying the result obtained by both Yong \cite{Yong} and
Wu \cite{Wu}, but with some additional ideas, we use this transformation and
virtue of the logarithm transformed introduced by El Karoui \& Hamadene
\cite{El-Karoui-Hamadene} to solve this problem and driven the necessary as
well as sufficient optimality conditions of the type risk sensitive performance.

The results of this paper generalize all the previous works on the subject,
into FBSDE with jumps diffusion under the risk sensitive performance. we
combine between two important results the first one was such of Djehiche et al
\cite{BTT}, while the second was Chala \cite{Chala 02, Chala 03}, for more
details for the risk sensitive the readers can see the papers \cite{Shi-Wu2,
Tembin-Zhu-Basar} and references of therein.

The paper is organized as follows: In section 2, we give the precise problem
formulations, and introduce the risk-sensitive model, and give the various
assumptions used throughout this paper. In section 3, we shall study our
system of fully coupled forward backward SDE, the new approach method
transformation of the adjoint process is given and studied, SMP for
risk-neutral is given, which will be the main result in next section, we give
our first main result, the necessary optimality conditions for risk-sensitive
control problem under an additional hypothesis is established. In section 4,
The sufficient optimality conditions for risk-sensitive performance cost is
our second main result, is obtained under the convexity of the Hamiltonian
function. In section 5, we finished the paper by given an application, a
financial model of mean variance with risk-sensitive performance functional is
the best application for our problem. The conclusion and remarks is the last
section (section 6).

\section{Problem and settings}

In all what follows, we will be worked on the classical probability space
$\left(  \Omega,\mathcal{F},\left(  \mathcal{F}\right)  _{t\leq T}%
,\mathbb{P}\right)  ,$ such that $\mathcal{F}_{0}$ contains all the
$\mathbb{P-}$null sets, $\mathcal{F}_{T}=\mathcal{F}$ \ for an arbitrarily
fixed time horizon $T$, and $\left(  \mathcal{F}_{t}\right)  _{t\leq T}$
\ satisfies the usual conditions. We assume that the filtration $\left(
\mathcal{F}\right)  _{t\leq T}$ is generated by the following two mutually
independent processes

\begin{enumerate}
\item[(i)] $\left(  W\left(  t\right)  \right)  _{t\geq0}$ is a
one-dimensional standard Brownian motion$.$

\item[(ii)] Poisson random measure $N$ on $\left[  0,T\right]  \times\Gamma,$
where $\Gamma\subset\mathbb{R-}\left\{  0\right\}  .$ We denote by $\left(
\mathcal{F}_{t}^{W}\right)  _{t\leq T}$ ( resp. $\left(  \mathcal{F}_{t}%
^{N}\right)  _{t\leq T}$) the $\mathbb{P-}$augmentation of the natural
filtration of $W$ ( resp. $N$). Obviously, we have
\[
\mathcal{F}_{t}:=\sigma\left[  \int_{0}^{s}\int_{A}N\left(  d\lambda
,dr\right)  ;\text{ }s\leq t,\text{ }A\in\mathcal{B}\left(  \Gamma\right)
\right]  \vee\sigma\left[  W\left(  s\right)  ;\text{ }s\leq t\right]
\vee\mathcal{N},
\]

\end{enumerate}

where $\mathcal{N}$ contains all $\mathbb{P}-$null sets in $\mathcal{F}$, and
$\sigma_{1}\vee\sigma_{2}$ denotes the $\sigma-$field generated by $\sigma
_{1}\cup\sigma_{2}.$ We assume that the compensator of $N$ has the form
$\mu\left(  dt,d\lambda\right)  =m\left(  d\lambda\right)  dt,$ for some
positive and $\sigma-$finite L\'{e}vy measure $m$ on $\Gamma,$ endowed with
its Borel $\sigma-$field $\mathcal{B}\left(  \Gamma\right)  .$ We suppose that
$%
{\displaystyle\int_{\Gamma}}
1\wedge\left\vert \lambda\right\vert ^{2}m\left(  d\lambda\right)  <\infty,$
and write $\tilde{N}=N-mdt$ for the compensated jump martingale random measure
of $N.$

\begin{notation}
We need to define some additional notations. Given $s\leq t,$ let us introduce
the following spaces

$\mathcal{S}^{2}\left(  \left[  0,T\right]  ,\mathbb{R}\right)  $ the set of
$\mathbb{R}$- valued adapted cadl\`{a}g processes $P$ such that%
\[
\left\Vert P\right\Vert _{\mathcal{S}^{2}\left(  \left[  0,T\right]
,\mathbb{R}\right)  }:=\mathbb{E}\left[  \underset{r\in\left[  0,T\right]
}{\sup}\left\vert P\left(  r\right)  \right\vert ^{2}\right]  ^{\frac{1}{2}%
}<+\infty.
\]
$\mathcal{M}^{2}\left(  \left[  0,T\right]  ,\mathbb{R}\right)  $ is the set
of progressively measurable $\mathbb{R-}$valued processes $Q$ such that
\[
\left\vert \left\vert Q\right\vert \right\vert _{\mathcal{M}^{2}\left(
\left[  0,T\right]  ,\mathbb{R}\right)  }:=\mathbb{E}\left[  \int_{0}%
^{T}\left\vert Q\left(  r\right)  \right\vert ^{2}dr\right]  ^{\frac{1}{2}%
}<+\infty.
\]
$\mathcal{L}_{m}^{2}\left(  \left[  0,T\right]  ,\mathbb{R}\right)  $ is the
set of $\mathcal{B}\left(  \left[  0,T\right]  \times\Omega\right)
\otimes\mathcal{B}\left(  \Gamma\right)  $ measurable maps $R:\left[
0,T\right]  \times\Omega\times\Gamma\rightarrow\mathbb{R}$ such that
\[
\left\Vert R\right\Vert _{\mathcal{L}_{m}^{2}\left(  \left[  0,T\right]
,\mathbb{R}\right)  }:=\mathbb{E}\left[  \int_{0}^{T}%
{\displaystyle\int_{\Gamma}}
\left\vert R\left(  r\right)  \right\vert ^{2}m\left(  d\lambda\right)
dr\right]  ^{\frac{1}{2}}<+\infty,
\]

we denote by $\mathbb{E}$ the expectation with respect to $\mathbb{P}$

Let $T$ be a strictly positive real number and $U$ is a convex nonempty subset
of $\mathbb{R}.$
\end{notation}

\begin{definition}
Let $U$ be a nonempty closed subset in $\ \mathbb{R}.$ An admissible control
is a $U-$valued measurable $\mathcal{F}_{t}-$adapted process $v,$ such that
$\left\Vert v\right\Vert _{S^{2}}<\infty.$ We denote by $\ \mathcal{U}$ the
set of all admissible controls
\end{definition}

For all $v\in\mathcal{U}$, we consider the following fully coupled
forward-backward with jump system%
\begin{equation}
\left\{
\begin{array}
[c]{ll}%
dx\left(  t\right)  = & b\left(  t,x\left(  t\right)  ,y\left(  t\right)
,z\left(  t\right)  ,r\left(  t,.\right)  ,v\left(  t\right)  \right)  dt\\
& +\sigma\left(  t,x\left(  t\right)  ,y\left(  t\right)  ,z\left(  t\right)
,r\left(  t,.\right)  ,v\left(  t\right)  \right)  dW\left(  t\right) \\
& +%
{\displaystyle\int_{\Gamma}}
\gamma\left(  t-,x\left(  t-\right)  ,y\left(  t-\right)  ,z\left(  t-\right)
,r\left(  t-,\lambda\right)  ,v\left(  t-\right)  ,\lambda\right)  \tilde
{N}\left(  dt,d\lambda\right) \\
dy\left(  t\right)  = & -g\left(  t,x\left(  t\right)  ,y\left(  t\right)
,z\left(  t\right)  ,r\left(  t,.\right)  ,v\left(  t\right)  \right)
dt+z\left(  t\right)  dW\left(  t\right) \\
& +%
{\displaystyle\int_{\Gamma}}
r\left(  t,\lambda\right)  \tilde{N}\left(  dt,d\lambda\right) \\
x\left(  0\right)  = & d,\text{ \ }y\left(  T\right)  =a,\text{
\ \ \ \ \ \ \ \ \ \ \ \ \ \ \ }t\in\left[  0,T\right]
\end{array}
\right.  \label{EQ}%
\end{equation}

where $b:\left[  0,T\right]  \times%
\mathbb{R}
\times%
\mathbb{R}
\times%
\mathbb{R}
\times\Gamma\times\mathcal{U}\rightarrow%
\mathbb{R}
,$ $\sigma:\left[  0,T\right]  \times%
\mathbb{R}
\times%
\mathbb{R}
\times%
\mathbb{R}
\times\Gamma\times\mathcal{U}\rightarrow\mathcal{%
\mathbb{R}
}$, $g:\left[  0,T\right]  \times%
\mathbb{R}
\times%
\mathbb{R}
\times%
\mathbb{R}
\times\Gamma\times\mathcal{U}\rightarrow%
\mathbb{R}
.$ and $\gamma:\left[  0,T\right]  \times%
\mathbb{R}
\times%
\mathbb{R}
\times%
\mathbb{R}
\times\Gamma\times\mathcal{U}\boldsymbol{\longrightarrow}\mathbb{R}$ are given
maps. If $\left(  x\left(  .\right)  ,y\left(  .\right)  ,z\left(  .\right)
,r\left(  .,.\right)  \right)  $ is the unique solution of $\left(
\ref{EQ}\right)  $ associated with $v\left(  .\right)  \in\mathcal{U}$.

The functional cost of the risk-sensitive type is given by%
\begin{align}
&  J^{\theta}\left(  v\right) \label{J}\\
&  =\mathbb{E}\left(  \exp\theta\left[
{\displaystyle\int_{0}^{T}}
f\left(  t,x\left(  t\right)  ,y\left(  t\right)  ,z\left(  t\right)
,r\left(  t,.\right)  ,v\left(  t\right)  \right)  dt+\Phi\left(  x^{v}\left(
T\right)  \right)  +\Psi\left(  y^{v}\left(  0\right)  \right)  \right]
\right)  ,\nonumber
\end{align}

where $\Phi:%
\mathbb{R}
\rightarrow%
\mathbb{R}
,$ $\Psi:%
\mathbb{R}
\rightarrow%
\mathbb{R}
,$ $f:\left[  0,T\right]  \times%
\mathbb{R}
\times%
\mathbb{R}
\times%
\mathbb{R}
\times\Gamma\times\mathcal{U}\rightarrow%
\mathbb{R}
$ are given maps, and $\theta>0$ is called the risk-sensitive parameter.

Our risk-sensitive stochastic optimal control problem is stated as follows:
For given $\left(  t,x\left(  t\right)  ,y\left(  t\right)  ,z\left(
t\right)  ,r\left(  t,.\right)  \right)  \in\left[  0,T\right]  \times
\mathbb{R}^{4},$ minimize $\left(  \ref{J}\right)  $ subject to $\left(
\ref{EQ}\right)  $ over $\mathcal{U}.$%
\begin{equation}
\inf_{v\in\mathcal{U}}J^{\theta}\left(  v\right)  =J^{\theta}\left(  u\right)
. \label{inf}%
\end{equation}

A control that solves the problem $\left\{  \left(  \ref{EQ}\right)  ,\left(
\ref{J}\right)  ,\left(  \ref{inf}\right)  \right\}  $ is called optimal. Our
goal is to establish a necessary optimality conditions as well as a sufficient
optimality conditions, satisfied by a given optimal control, in the form of
stochastic maximum principle (SMP in short).

We give some notations $\Upsilon=\left(  x^{v}\left(  t\right)  ,y^{v}\left(
t\right)  ,z^{v}\left(  t\right)  ,r^{v}\left(  t,.\right)  \right)  ^{\top}$,
where $\left(  .\right)  ^{\top}$ denotes the transport of the matrix,

and $M\left(  t,\Upsilon\right)  =\left(
\begin{tabular}
[c]{c}%
$b$\\
$\sigma$\\
$-g$%
\end{tabular}
\ \right)  \left(  t,\Upsilon\right)  .$

We introduce the following assumptions.

$\mathbf{H}_{1}:$

For each $\Upsilon\in%
\mathbb{R}
\times%
\mathbb{R}
\times%
\mathbb{R}
$, $M\left(  t,\Upsilon\right)  $ is an $\mathcal{F}_{t}-$measurable process
defined on $[0,T]$ with $M\left(  t,\Upsilon\right)  \in\mathcal{M}^{2}\left(
\left[  0,T\right]  ;%
\mathbb{R}
\times%
\mathbb{R}
\times%
\mathbb{R}
\times\Gamma\right)  .$

$\mathbf{H}_{2}:$

$M(t,.)$ satisfies Lipschitz conditions$:$ There exists a constant $k>0,$ such
that%
\[
\left\vert M\left(  t,\Upsilon\right)  -M\left(  t,\Upsilon^{\prime}\right)
\right\vert \leq k\left\vert \Upsilon-\Upsilon^{\prime}\right\vert
\forall\Upsilon,\Upsilon^{\prime}\in%
\mathbb{R}
\times%
\mathbb{R}
\times%
\mathbb{R}
\times\Gamma,\forall t\in\lbrack0,T].
\]

The following monotonic conditions introduced in \cite{Wu}, are the main
assumptions in this paper.

$\mathbf{H}_{3}:$

$\left\langle M\left(  t,\Upsilon\right)  -M\left(  t,\Upsilon^{\prime
}\right)  ,\Upsilon-\Upsilon\right\rangle \leq\beta\left\vert \Upsilon
-\Upsilon^{\prime}\right\vert ^{2},$ for every\ $\Upsilon=\left(
x,y,z,r\right)  ^{\ast}$ and$\ \Upsilon^{\prime}=\left(  x^{\prime},y^{\prime
},z^{\prime},r^{\prime}\right)  ^{\ast}\in%
\mathbb{R}
\times%
\mathbb{R}
\times%
\mathbb{R}
\times\Gamma$, $\forall$ $t\in\lbrack0,T],$ where $\beta$ is a positive
constant.\textbf{\ }

$\mathcal{U}$ \ is a convex subset of $\mathbb{R}.$

\begin{proposition}
For any given admissible control $v\left(  .\right)  $ and\ under the
assumptions $(\mathbf{H}_{\mathbf{1}})$, $(\mathbf{H}_{\mathbf{2}})$ and
$(\mathbf{H}_{\mathbf{3}})$, the fully coupled FBSDE with jump diffusion
$\left(  \ref{EQ}\right)  \ $admits an unique solution

$\left(  x^{v}\left(  t\right)  ,y^{v}\left(  t\right)  ,z^{v}\left(
t\right)  ,r^{v}\left(  t,.\right)  \right)  \in\left(  \mathcal{M}^{2}\left(
\left[  0,T\right]  ;%
\mathbb{R}
\times%
\mathbb{R}
\times\Gamma\right)  \right)  ^{2}\times\mathcal{S}^{2}\left(  \left[
0,T\right]  ;%
\mathbb{R}
\times\Gamma\right)  .$
\end{proposition}

\begin{proof}
The proof can be seen in \cite{Wu}.
\end{proof}

Next, we assume that

$\mathbf{H}_{4}:$

$i)b,$ $\sigma,$ $g,$ $f,$ $\Phi$ and $\Psi$ are continuously differentiable
with respect to $\left(  x^{v},y^{v},z^{v},r^{v}\left(  .\right)  \right)  .$

$ii)$ All the derivatives of $b,$ $\sigma,$ $g$ and $f$ are bounded by

$C\left(  1+\left\vert x^{v}\right\vert +\left\vert y^{v}\right\vert
+\left\vert z^{v}\right\vert +\left\vert v\right\vert +\left\vert
r^{v}\right\vert \right)  .$

$iii)$ The derivatives of $\Phi,$ $\Psi$ are bounded by $C\left(  1+\left\vert
x^{v}\right\vert \right)  $ and $C\left(  1+\left\vert y^{v}\right\vert
\right)  $ respectively.

Under the above assumptions, for every $v\in\mathcal{U}$ equation $\left(
\ref{EQ}\right)  $ has a unique strong solution and the function cost
$J^{\theta}$ is well defined from $\mathcal{U}$ into $%
\mathbb{R}
$.

\section{Necessary optimality conditions and auxiliary process}

First of all, we may introduce an auxiliary state process $\xi^{v}\left(
t\right)  $ which is solution of the following stochastic differential
equation (SDE in short):%
\[
d\xi^{v}\left(  t\right)  =f\left(  t,x^{v}\left(  t\right)  ,y^{v}\left(
t\right)  ,z^{v}\left(  t\right)  ,r^{v}\left(  t,.\right)  ,v\left(
t\right)  \right)  dt,\text{\ }\xi^{v}\left(  0\right)  =0.
\]

From the above auxiliary process, the fully coupled forward-backward type
control problem is equivalent to%
\begin{equation}
\left\{
\begin{array}
[c]{ll}%
\inf\limits_{v\in\mathcal{U}} & \mathbb{E}\left[  \exp\theta\left\{
\Phi\left(  x^{v}\left(  T\right)  \right)  +\Psi\left(  y^{v}\left(
0\right)  \right)  +\xi\left(  T\right)  \right\}  \right]  ,\\
\text{subject to} & \\
d\xi^{v}\left(  t\right)  = & f\left(  t,x^{v}\left(  t\right)  ,y^{v}\left(
t\right)  ,z^{v}\left(  t\right)  ,r^{v}\left(  t,.\right)  ,v\left(
t\right)  \right)  dt,\\
dx^{v}\left(  t\right)  = & b\left(  t,x^{v}\left(  t\right)  ,y^{v}\left(
t\right)  ,z^{v}\left(  t\right)  ,r^{v}\left(  t,.\right)  ,v\left(
t\right)  \right)  dt\\
& +\sigma\left(  t,x^{v}\left(  t\right)  ,y^{v}\left(  t\right)
,z^{v}\left(  t\right)  ,r^{v}\left(  t,.\right)  ,v\left(  t\right)  \right)
dW\left(  t\right) \\
& +%
{\displaystyle\int_{\Gamma}}
\gamma\left(  t,x\left(  t-\right)  ,y\left(  t-\right)  ,z\left(  t-\right)
,r\left(  t-,\lambda\right)  ,v\left(  t-\right)  ,\lambda\right)  \tilde
{N}\left(  dt,d\lambda\right)  ,\\
dy^{v}\left(  t\right)  = & -g\left(  t,x^{v}\left(  t\right)  ,y^{v}\left(
t\right)  ,z^{v}\left(  t\right)  ,r^{v}\left(  t,.\right)  ,v\left(
t\right)  \right)  dt+z^{v}\left(  t\right)  dW\left(  t\right) \\
& +%
{\displaystyle\int_{\Gamma}}
r^{v}\left(  t,\lambda\right)  \tilde{N}\left(  dt,d\lambda\right)  ,\\
\xi^{v}\left(  0\right)  = & 0,\text{ }x^{v}\left(  0\right)  =d\text{,\ }%
y^{v}\left(  T\right)  =a.
\end{array}
\right.  \label{problem 1}%
\end{equation}

We denote by%
\[
A_{T}^{\theta}:=\exp\theta\left\{  \Phi\left(  x^{v}\left(  T\right)  \right)
+\Psi\left(  y^{v}\left(  0\right)  \right)  +%
{\displaystyle\int_{0}^{T}}
f\left(  t,x^{v}\left(  t\right)  ,y^{v}\left(  t\right)  ,z^{v}\left(
t\right)  ,r^{v}\left(  t,.\right)  ,v\left(  t\right)  \right)  dt\right\}
,
\]
and we can put also
\[
\Theta_{T}:=\Phi\left(  x^{v}\left(  T\right)  \right)  +\Psi\left(
y^{v}\left(  0\right)  \right)  +%
{\displaystyle\int_{0}^{T}}
f\left(  t,x^{v}\left(  t\right)  ,y^{v}\left(  t\right)  ,z^{v}\left(
t\right)  ,r^{v}\left(  t,.\right)  ,v\left(  t\right)  \right)  dt,
\]

the risk-sensitive loss functional is given by%
\begin{align*}
\Theta_{\theta}  &  :=\frac{1}{\theta}\log\mathbb{E}\left(  \exp\left\{
\Phi\left(  x^{v}\left(  T\right)  \right)  +\Psi\left(  y^{v}\left(
0\right)  \right)  \right.  \right. \\
&  \left.  \left.  +%
{\displaystyle\int_{0}^{T}}
f\left(  t,x^{v}\left(  t\right)  ,y^{v}\left(  t\right)  ,z^{v}\left(
t\right)  ,r^{v}\left(  t,.\right)  ,v\left(  t\right)  \right)  dt\right\}
\right) \\
&  =\frac{1}{\theta}\log\mathbb{E}\left(  \exp\left\{  \theta\Theta
_{T}\right\}  \right)  .
\end{align*}

When the risk-sensitive index $\theta$ is small, the functional $\Theta
_{\theta}$ can be expanded as $\mathbb{E}\left(  \Theta_{T}\right)
+\frac{\theta}{2}Var\left(  \Theta_{T}\right)  +O\left(  \theta^{2}\right)  ,$
where, $Var\left(  \Theta_{T}\right)  $ denotes the variance of $\Theta_{T}.$
If $\theta<0,$ the variance of $\Theta_{T},$ as a measure of risk, improves
the performance $\Theta_{\theta}$, in which case the optimizer is called
\textit{risk seeker. }But, when $\theta>0,$ the variance of $\Theta_{T}$
worsens the performance $\Theta_{\theta}$, in which case the optimizer is
called \textit{risk averse}. The risk-neutral loss functional \textit{\ }%
$\mathbb{E}\left(  \Theta_{T}\right)  $ can be seen as a limit of
risk-sensitive functional $\Theta_{\theta}$ when $\theta\rightarrow0$, for
more details the reader can see the papers $\cite{chala hefayed khallout}.$

\begin{notation}
\label{Notation}We will use the following notation throughout this
paper.\newline For $\phi\in\left\{  b,\sigma,f,g,H^{\theta},\widetilde
{H}^{\theta}\right\}  $, we define%
\[
\left\{
\begin{array}
[c]{l}%
\phi\left(  t\right)  =\phi\left(  t,x^{v}\left(  t\right)  ,y^{v}\left(
t\right)  ,z^{v}\left(  t\right)  ,r^{v}\left(  t,.\right)  ,v\left(
t\right)  \right)  ,\\
\partial\phi\left(  t\right)  =\phi\left(  t,x^{v}\left(  t\right)
,y^{v}\left(  t\right)  ,z^{v}\left(  t\right)  ,r^{v}\left(  t,.\right)
,v\left(  t\right)  \right) \\
-\phi\left(  t,x^{v}\left(  t\right)  ,y^{v}\left(  t\right)  ,z^{v}\left(
t\right)  ,r^{v}\left(  t,.\right)  ,u\left(  t\right)  \right)  ,\\
\phi_{\zeta}\left(  t\right)  =\frac{\partial\phi}{\partial\zeta}\left(
t,x^{v}\left(  t\right)  ,y^{v}\left(  t\right)  ,z^{v}\left(  t\right)
,r^{v}\left(  t,.\right)  ,v\left(  t\right)  \right)  ,\text{ }\zeta=x,\text{
}y,\text{ }z,\text{ }r\left(  .\right)  .
\end{array}
\right.
\]

and $\gamma\left(  t-,\lambda\right)  $ it means that the function $\gamma$ is c\`{a}dlag.
\end{notation}

Where $v_{t}$ in an admissible control from $\mathcal{U}$.

We assume that $(\mathbf{H}_{\mathbf{1}})$, $(\mathbf{H}_{\mathbf{2}%
}),(\mathbf{H}_{\mathbf{3}})$ and $(\mathbf{H}_{\mathbf{4}})$ hold, we might
apply the SMP for risk-neutral of fully coupled forward-backward type control
from Yong $\cite{Yong},$ to augmented state dynamics $\left(  \xi
,x,y,z,r\right)  $ and derive the adjoint equation. There exist unique
$\mathcal{F}_{t}-$adapted of processes $\left(  p_{1},q_{1},\pi_{1}\right)
,\left(  p_{2},q_{2},\pi_{2}\right)  ,\left(  p_{3},q_{3},\pi_{3}\right)  ,$
which solve the following system matrix of backward SDEs%
\begin{equation}
\left\{
\begin{array}
[c]{ll}%
d\overrightarrow{p}\left(  t\right)  & =\left(
\begin{tabular}
[c]{l}%
$dp_{1}\left(  t\right)  $\\
$dp_{2}\left(  t\right)  $\\
$dp_{3}\left(  t\right)  $%
\end{tabular}
\ \ \ \right) \\
& =-\left(
\begin{tabular}
[c]{lll}%
$0$ & $0$ & $0$\\
$f_{x}\left(  t\right)  $ & $b_{x}\left(  t\right)  $ & $g_{x}\left(
t\right)  $\\
$f_{y}\left(  t\right)  $ & $b_{y}\left(  t\right)  $ & $g_{y}\left(
t\right)  $%
\end{tabular}
\ \ \ \right)  \left(
\begin{tabular}
[c]{l}%
$p_{1}\left(  t\right)  $\\
$p_{2}\left(  t\right)  $\\
$p_{3}\left(  t\right)  $%
\end{tabular}
\ \ \ \right)  dt\\
& -\left(
\begin{tabular}
[c]{lll}%
$0$ & $0$ & $0$\\
$0$ & $\sigma_{x}\left(  t\right)  $ & $0$\\
$0$ & $\sigma_{y}\left(  t\right)  $ & $0$%
\end{tabular}
\ \ \ \right)  \left(
\begin{tabular}
[c]{l}%
$q_{1}\left(  t\right)  $\\
$q_{2}\left(  t\right)  $\\
$q_{3}\left(  t\right)  $%
\end{tabular}
\ \ \ \right)  dt\\
& +%
{\displaystyle\int_{\Gamma}}
\left(
\begin{tabular}
[c]{lll}%
$0$ & $0$ & $0$\\
$0$ & $\gamma_{x}\left(  t-,\lambda\right)  $ & $0$\\
$0$ & $\gamma_{y}\left(  t-,\lambda\right)  $ & $0$%
\end{tabular}
\ \ \ \right)  \left(
\begin{tabular}
[c]{l}%
$\pi_{1}\left(  t,\lambda\right)  $\\
$\pi_{2}\left(  t,\lambda\right)  $\\
$\pi_{3}\left(  t,\lambda\right)  $%
\end{tabular}
\ \ \ \right)  m\left(  d\lambda\right)  dt\\
& +\left(
\begin{tabular}
[c]{l}%
$q_{1}\left(  t\right)  $\\
$q_{2}\left(  t\right)  $\\
$q_{3}\left(  t\right)  $%
\end{tabular}
\ \right)  dW\left(  t\right)  +%
{\displaystyle\int_{\Gamma}}
\left(
\begin{tabular}
[c]{l}%
$\pi_{1}\left(  t,\lambda\right)  $\\
$\pi_{2}\left(  t,\lambda\right)  $\\
$\pi_{3}\left(  t,\lambda\right)  $%
\end{tabular}
\ \right)  \tilde{N}\left(  dt,d\lambda\right) \\
\left(
\begin{tabular}
[c]{l}%
$p_{1}\left(  T\right)  $\\
$p_{2}\left(  T\right)  $%
\end{tabular}
\ \ \ \ \ \ \ \right)  & =\theta A_{T}\left(
\begin{tabular}
[c]{l}%
$1$\\
$\Phi_{x}\left(  x_{T}^{u}\right)  $%
\end{tabular}
\ \right) \\
p_{3}\left(  0\right)  & =\theta\Psi_{y}\left(  y^{u}\left(  0\right)
\right)  A_{T},
\end{array}
\right.  \label{adj1}%
\end{equation}

with $\mathbb{E}\left[
{\displaystyle\sum\limits_{i=1}^{3}}
\sup\limits_{0\leq t\leq T}\left\vert p_{i}\left(  t\right)  \right\vert ^{2}+%
{\displaystyle\sum\limits_{i=1}^{2}}
{\displaystyle\int_{0}^{T}}
\left\vert q_{i}\left(  t\right)  \right\vert ^{2}dt\right]  <\infty,$ and

$q_{3}\left(  t\right)  =-Tr\left[  \left(
\begin{array}
[c]{cc}%
f_{z}\left(  t\right)  & b_{z}\left(  t\right) \\
\sigma_{z}\left(  t\right)  & g_{z}\left(  t\right)
\end{array}
\right)  \left(
\begin{array}
[c]{cc}%
p_{1}\left(  t\right)  & q_{2}\left(  t\right) \\
p_{2}\left(  t\right)  & p_{3}\left(  t\right)
\end{array}
\right)  \right]  +%
{\displaystyle\int_{\Gamma}}
\gamma_{z}\left(  t-,\lambda\right)  \pi_{2}\left(  t,\lambda\right)  m\left(
d\lambda\right)  ,$

$\pi_{3}\left(  t,\lambda\right)  =-Tr\left[  \left(
\begin{array}
[c]{cc}%
f_{r}\left(  t\right)  & b_{r}\left(  t\right) \\
\sigma_{r}\left(  t\right)  & g_{r}\left(  t\right)
\end{array}
\right)  \left(
\begin{array}
[c]{cc}%
p_{1}\left(  t\right)  & q_{2}\left(  t\right) \\
p_{2}\left(  t\right)  & p_{3}\left(  t\right)
\end{array}
\right)  \right]  +%
{\displaystyle\int_{\Gamma}}
\gamma_{r}\left(  t-,\lambda\right)  \pi_{2}\left(  t,\lambda\right)  m\left(
d\lambda\right)  .$

To this end we may define $\left(  \ref{adj1}\right)  $ in the compact form
as
\[
\left\{
\begin{array}
[c]{l}%
d\overrightarrow{p}\left(  t\right)  =\left(
\begin{tabular}
[c]{l}%
$dp_{1}\left(  t\right)  $\\
$dp_{2}\left(  t\right)  $\\
$dp_{3}\left(  t\right)  $%
\end{tabular}
\ \right)  =-F\left(  t\right)  dt+\Sigma\left(  t\right)  dW\left(  t\right)
+%
{\displaystyle\int_{\Gamma}}
R\left(  t,\lambda\right)  \tilde{N}\left(  dt,d\lambda\right) \\
\left(
\begin{tabular}
[c]{l}%
$p_{1}\left(  T\right)  $\\
$p_{2}\left(  T\right)  $%
\end{tabular}
\ \right)  =\theta A_{T}\left(
\begin{tabular}
[c]{l}%
$1$\\
$\Phi_{x}\left(  x_{T}^{u}\right)  $%
\end{tabular}
\ \right)  ,\text{ and }p_{3}\left(  0\right)  =\theta\Psi_{y}\left(
y^{u}\left(  0\right)  \right)  A_{T},
\end{array}
\right.
\]

where
\[%
\begin{array}
[c]{ll}%
F\left(  t\right)  = & \left(
\begin{tabular}
[c]{lll}%
$0$ & $0$ & $0$\\
$f_{x}\left(  t\right)  $ & $b_{x}\left(  t\right)  $ & $g_{x}\left(
t\right)  $\\
$f_{y}\left(  t\right)  $ & $b_{y}\left(  t\right)  $ & $g_{y}\left(
t\right)  $%
\end{tabular}
\ \right)  \left(
\begin{tabular}
[c]{l}%
$p_{1}\left(  t\right)  $\\
$p_{2}\left(  t\right)  $\\
$p_{3}\left(  t\right)  $%
\end{tabular}
\ \right) \\
& +\left(
\begin{tabular}
[c]{lll}%
$0$ & $0$ & $0$\\
$0$ & $\sigma_{x}\left(  t\right)  $ & $0$\\
$0$ & $\sigma_{y}\left(  t\right)  $ & $0$%
\end{tabular}
\ \right)  \left(
\begin{tabular}
[c]{l}%
$q_{1}\left(  t\right)  $\\
$q_{2}\left(  t\right)  $\\
$q_{3}\left(  t\right)  $%
\end{tabular}
\ \right) \\
& -%
{\displaystyle\int_{\Gamma}}
\left(
\begin{tabular}
[c]{lll}%
$0$ & $0$ & $0$\\
$0$ & $\gamma_{x}\left(  t-,\lambda\right)  $ & $0$\\
$0$ & $\gamma_{y}\left(  t-,\lambda\right)  $ & $0$%
\end{tabular}
\ \right)  \left(
\begin{tabular}
[c]{l}%
$\pi_{1}\left(  t,\lambda\right)  $\\
$\pi_{2}\left(  t,\lambda\right)  $\\
$\pi_{3}\left(  t,\lambda\right)  $%
\end{tabular}
\ \right)  m\left(  d\lambda\right)  ,
\end{array}
\]%
\[%
\begin{array}
[c]{cc}%
\Sigma\left(  t\right)  = & \left(
\begin{tabular}
[c]{l}%
$q_{1}\left(  t\right)  $\\
$q_{2}\left(  t\right)  $\\
$q_{3}\left(  t\right)  $%
\end{tabular}
\ \right)  ,
\end{array}
\]
and%
\[%
\begin{array}
[c]{ll}%
R\left(  t,.\right)  = & \left(
\begin{tabular}
[c]{l}%
$\pi_{1}\left(  t,.\right)  $\\
$\pi_{2}\left(  t,.\right)  $\\
$\pi_{3}\left(  t,.\right)  $%
\end{tabular}
\ \right)  .
\end{array}
\]

We suppose here that $\widetilde{H}^{\theta}$ be the Hamiltonian associated
with the optimal state dynamics $\left(  \xi^{u},x^{u},y^{u},z^{u}%
,r^{u}(.\right)  ),$ and the triplet of adjoint processes $\left(
\overrightarrow{p}\left(  t\right)  ,\overrightarrow{q}\left(  t\right)
,\overrightarrow{\pi}\left(  t,.\right)  \right)  $ is given by%
\[%
\begin{array}
[c]{l}%
\widetilde{H}^{\theta}\left(  t,\xi^{u}\left(  t\right)  ,x^{u}\left(
t\right)  ,y^{u}\left(  t\right)  ,z^{u}\left(  t\right)  ,r\left(
t,.\right)  ,u\left(  t\right)  ,\overrightarrow{p}\left(  t\right)
,\overrightarrow{q}\left(  t\right)  ,\overrightarrow{\pi}\left(  t,.\right)
\right) \\
=\left(
\begin{tabular}
[c]{l}%
$f\left(  t\right)  $\\
$b\left(  t\right)  $\\
$g\left(  t\right)  $%
\end{tabular}
\ \right)  \left(  \overrightarrow{p}\left(  t\right)  \right)  ^{\top
}+\left(
\begin{tabular}
[c]{l}%
$0$\\
$\sigma\left(  t\right)  $\\
$0$%
\end{tabular}
\ \right)  \left(  \overrightarrow{q}\left(  t\right)  \right)  ^{\top}\\
-%
{\displaystyle\int_{\Gamma}}
\left(
\begin{tabular}
[c]{l}%
$0$\\
$\gamma\left(  t-,\lambda\right)  $\\
$0$%
\end{tabular}
\ \right)  \left(  \overrightarrow{\pi}\left(  t,\lambda\right)  \right)
^{\top}m\left(  d\lambda\right)  .
\end{array}
\]

\begin{theorem}
\label{theoriskneutral}Assume that $(\mathbf{H}_{\mathbf{1}})$, $(\mathbf{H}%
_{\mathbf{2}}),(\mathbf{H}_{\mathbf{3}})$ and $(\mathbf{H}_{\mathbf{4}})$ hold.

If $\left(  \xi^{u}\left(  .\right)  ,x^{u}\left(  .\right)  ,y^{u}\left(
.\right)  ,z^{u}\left(  .\right)  ,r(.,.\right)  )$ is an optimal solution of
the risk-neutral control problem $\left(  \ref{problem 1}\right)  ,$ then
there exist $\mathcal{F}_{t}-$adapted processes

$\left(  \left(  p_{1},q_{1},\pi_{1}\right)  ,\left(  p_{2},q_{2},\pi
_{2}\right)  ,\left(  p_{3},q_{3},\pi_{3}\right)  \right)  $ that satisfy
$\left(  \ref{adj1}\right)  ,$ such that%
\begin{equation}
\widetilde{H}_{v}^{\theta}\left(  t\right)  \left(  u_{t}-v_{t}\right)  \geq0,
\label{SMP}%
\end{equation}
for all $u\in\mathcal{U},$ almost every $t$ and $\mathbb{P}-$almost surely,
where $\widetilde{H}_{v}^{\theta}\left(  t\right)  $ is defined in notation
$\left(  \ref{Notation}\right)  $.
\end{theorem}

\begin{proof}
For more details the reader can see paper $\cite{Yong}$ with the result of
paper $\cite{SW2}.$
\end{proof}

\subsection{Expected Exponential Utility}

The expected exponential utility can be transformed into quadratic BSDE, this
Backward stochastic differential equation it permets us to find an other way
to resoudre the problem of adjoint equation which play a good rule in the
component of the Hamiltonian function.

As we said, Theorem $\ref{theoriskneutral}$ is a good SMP for the risk-neutral
of forward backward control problem. We follow the same approach used in
$\cite{Chala 02, BTT},$ and suggest a transformation of the adjoint processes
$\left(  p_{1},q_{1},\pi_{1}\left(  .\right)  \right)  ,$ $\left(  p_{2}%
,q_{2},\pi_{2}\left(  .\right)  \right)  ,$ $\left(  p_{3},q_{3,}\pi
_{3}\left(  .\right)  \right)  $ in such a way to omit the first component
$\left(  p_{1},q_{1},\pi_{1}\left(  .\right)  \right)  $ in $\left(
\ref{adj1}\right)  ,$ and to obtain the SMP $\left(  \ref{SMP}\right)  $ in
terms of only the last two adjoint processes, that we denote them by $\left(
\left(  \widetilde{p}_{2},\widetilde{q}_{2},\widetilde{\pi}_{2}\left(
.\right)  \right)  ,\left(  \widetilde{p}_{3},\widetilde{q}_{3},,\widetilde
{\pi}_{3}\left(  .\right)  \right)  \right)  $. Noting that $dp_{1}\left(
t\right)  =q_{1}\left(  t\right)  dW_{t}+%
{\displaystyle\int_{\Gamma}}
\pi_{1}\left(  t,\lambda\right)  \tilde{N}\left(  dt,d\lambda\right)  $ and
$p_{1}\left(  T\right)  =\theta A_{T}^{\theta},$ the explicit solution of this
backward SDE is%
\begin{equation}
p_{1}\left(  t\right)  =\theta\mathbb{E}\left[  A_{T}^{\theta}\left\vert
\text{ }\right.  \mathcal{F}_{t}\right]  =\theta V^{\theta}\left(  t\right)  ,
\label{p1}%
\end{equation}

where%
\begin{equation}
V^{\theta}\left(  t\right)  :=\mathbb{E}\left[  A_{T}^{\theta}\left\vert
\text{ }\right.  \mathcal{F}_{t}\right]  ,\text{\ }0\leq t\leq T.
\label{V theta}%
\end{equation}

As a good look of $\left(  \ref{p1}\right)  ,$ it would be natural to choose a
transformation of $\left(  \widetilde{p},\widetilde{q},\widetilde{\pi}\left(
.\right)  \right)  $ instead of $\left(  \overrightarrow{p},\overrightarrow
{q},\overrightarrow{\pi}\left(  .\right)  \right)  $ $,$ where $\widetilde
{p}_{1}\left(  t\right)  =\dfrac{1}{\theta V^{\theta}\left(  t\right)  }%
p_{1}\left(  t\right)  =1.$

We consider the following transform%
\begin{equation}
\widetilde{p}\left(  t\right)  =\left(
\begin{tabular}
[c]{l}%
$\widetilde{p}_{1}\left(  t\right)  $\\
$\widetilde{p}_{2}\left(  t\right)  $\\
$\widetilde{p}_{3}\left(  t\right)  $%
\end{tabular}
\right)  :=\frac{1}{\theta V^{\theta}\left(  t\right)  }\overrightarrow
{p}\left(  t\right)  ,\text{\ }0\leq t\leq T. \label{p transformed}%
\end{equation}

By using $\left(  \ref{adj1}\right)  $ and $\left(  \ref{p transformed}%
\right)  ,$ we have%
\[
\widetilde{p}\left(  T\right)  :=\left(
\begin{tabular}
[c]{l}%
$\widetilde{p}_{1}\left(  T\right)  $\\
$\widetilde{p}_{2}\left(  T\right)  $%
\end{tabular}
\right)  =\left(
\begin{tabular}
[c]{l}%
$1$\\
$\Phi_{x}\left(  x^{u}\left(  T\right)  \right)  $%
\end{tabular}
\right)  ,\text{ and }\widetilde{p}\left(  0\right)  =\Psi_{y}\left(
y^{u}\left(  0\right)  \right)  .
\]

The following properties of the generic martingale $V^{\theta}$ are essential
in order to investigate the properties of these new processes $\left(
\widetilde{p}\left(  t\right)  ,\widetilde{q}\left(  t\right)  ,\widetilde
{\pi}\left(  t,.\right)  \right)  .$

In this part, we want to prove the relationship between the exponential
utility and the backward quadratic stochastic equation. First of all, it's
very important to write the expected exponential utility under this form%
\begin{equation}
e^{Y_{t}}=\mathbb{E}\left[  A_{t,T}\left\vert \text{ }\right.  \mathcal{F}%
_{t}\right]  =\mathbb{E}\left[  \exp\theta\left[
{\displaystyle\int_{t}^{T}}
f\left(  s,x\left(  s\right)  ,y\left(  s\right)  ,z\left(  s\right)
,r\left(  s,.\right)  ,v\left(  s\right)  \right)  ds+\Phi\left(  x^{v}\left(
T\right)  \right)  +\Psi\left(  y^{v}\left(  0\right)  \right)  \right]
\left\vert \text{ }\right.  \mathcal{F}_{t}\right]  .\label{EEU}%
\end{equation}

For more details about the Expected exponential utility optimization, the
reader can visits the papers $\cite{Dahl}$ and $\cite{Hu et al}.$

\begin{lemma}
\label{NSCEEU}The necessary and sufficient condition for the expected
exponential utility $\left(  \ref{EEU}\right)  $ is the backward quadratic
stochastic equation%
\begin{align}
&  \exp\left\{  \theta\Lambda^{\theta}\left(  t\right)  \right\}
\label{Elkaroui hamadene}\\
&  =\mathbb{E}\left[  \exp\theta\left[
{\displaystyle\int_{t}^{T}}
f\left(  t,x\left(  t\right)  ,y\left(  t\right)  ,z\left(  t\right)
,r\left(  t,.\right)  ,v\left(  t\right)  \right)  dt+\Phi\left(  x^{v}\left(
T\right)  \right)  +\Psi\left(  y^{v}\left(  0\right)  \right)  \right]
\left\vert \text{ }\right.  \mathcal{F}_{t}\right]  \Leftrightarrow\nonumber\\
&  \left\{
\begin{array}
[c]{ll}%
d\Lambda^{\theta}\left(  t\right)  = & -\left\{  f\left(  t\right)
+\frac{\theta}{2}\left\vert l\left(  t\right)  \right\vert ^{2}+\frac{\theta
}{2}%
{\displaystyle\int_{\Gamma}}
\left\vert L\left(  t,\lambda\right)  \right\vert ^{2}m(d\lambda)\right.  \\
& \left.  +%
{\displaystyle\int_{\Gamma}}
\left(  \dfrac{\exp\left(  \theta r\left(  t,\lambda\right)  \right)
-1}{\theta}-r\left(  t,\lambda\right)  \right)  m(d\lambda)\right\}
dt+l\left(  t\right)  dW\left(  t\right)  \\
& -%
{\displaystyle\int_{\Gamma}}
\left\{  \dfrac{\exp\left(  \theta r\left(  t,\lambda\right)  \right)
-1}{\theta}\right\}  \tilde{N}\left(  dt,d\lambda\right)  +%
{\displaystyle\int_{\Gamma}}
L\left(  t,\lambda\right)  \tilde{N}\left(  dt,d\lambda\right)  ,\\
\Lambda^{\theta}\left(  T\right)  = & \Phi_{x}\left(  x^{u}\left(  T\right)
\right)  +\Psi\left(  y^{u}\left(  0\right)  \right)  ,
\end{array}
\right.  \nonumber
\end{align}
where%
\[
\mathbb{E}\left[
{\displaystyle\int_{0}^{T}}
\left\vert l\left(  t\right)  \right\vert ^{2}dt+%
{\displaystyle\int_{0}^{T}}
{\displaystyle\int_{\Gamma}}
\left\vert L\left(  t,\lambda\right)  \right\vert ^{2}m(d\lambda)dt\right]
<\infty.
\]

\end{lemma}

\begin{proof}
We assume that $\left(  \ref{EEU}\right)  $ holds, the we get%
\begin{align*}
&  \exp\left\{  \theta\Lambda^{\theta}\left(  t\right)  +\theta%
{\displaystyle\int_{t}^{T}}
f\left(  s,x\left(  s\right)  ,y\left(  s\right)  ,z\left(  s\right)
,r\left(  s,.\right)  ,v\left(  s\right)  \right)  ds\right\}  \\
&
\begin{array}
[c]{l}%
=\mathbb{E}\left[  \exp\theta\left[
{\displaystyle\int_{t}^{T}}
f\left(  s,x\left(  s\right)  ,y\left(  s\right)  ,z\left(  s\right)
,r\left(  s,.\right)  ,v\left(  s\right)  \right)  ds+\right.  \right.  \\
\left.  \left.
{\displaystyle\int_{0}^{t}}
f\left(  s,x\left(  s\right)  ,y\left(  s\right)  ,z\left(  s\right)
,r\left(  s,.\right)  ,v\left(  s\right)  \right)  ds+\right.  \right.  \\
\left.  \left.  \Phi\left(  x^{v}\left(  T\right)  \right)  +\Psi\left(
y^{v}\left(  0\right)  \right)  \right]  \left\vert \text{ }\right.
\mathcal{F}_{t}\right]  \\
=\mathbb{E}\left[  \exp\theta\left(
{\displaystyle\int_{0}^{T}}
f\left(  s,x\left(  s\right)  ,y\left(  s\right)  ,z\left(  s\right)
,r\left(  s,.\right)  ,v\left(  s\right)  \right)  ds+\right.  \right.  \\
\left.  \left.  \Phi\left(  x^{v}\left(  T\right)  \right)  +\Psi\left(
y^{v}\left(  0\right)  \right)  \right)  \left\vert \text{ }\right.
\mathcal{F}_{t}\right]  \\
=\mathbb{E}\left[  \theta A_{T}^{\theta}\left\vert \text{ }\right.
\mathcal{F}_{t}\right]  .
\end{array}
\end{align*}

By using of martingale representation theorem, there exist a process square
integrable $Z$ with respect to norm $\left\vert \left\vert Q\right\vert
\right\vert _{\mathcal{M}^{2}\left(  \left[  0,T\right]  ,\mathbb{R}\right)
},$ and the process $r\left(  t,\lambda\right)  $ in the space $\mathcal{L}%
_{m}^{2}\left(  \left[  0,T\right]  ,\mathbb{R}\right)  ,$ puting
$\mathbb{E}\left[  A_{T}^{\theta}\right]  =\exp\left\{  \theta\Lambda^{\theta
}\left(  0\right)  \right\}  ,$ we get%
\[
\exp\left\{  \theta\Lambda^{\theta}\left(  t\right)  \right\}  -\exp\left\{
\theta\Lambda^{\theta}\left(  0\right)  \right\}  =\theta\int_{0}^{t}Z\left(
s\right)  dW\left(  s\right)  +\int_{0}^{t}%
{\displaystyle\int_{\Gamma}}
r\left(  t,\lambda\right)  \tilde{N}\left(  ds,d\lambda\right)  .
\]

By applying L\'{e}vy-Ito's formula to $\left(  \exp\left\{  \theta
\Lambda^{\theta}\left(  t\right)  +\theta%
{\displaystyle\int_{0}^{T}}
f\left(  s\right)  ds\right\}  \right)  $, we get%
\begin{equation}%
\begin{array}
[c]{l}%
\theta\left[  \theta\Lambda^{\theta}\left(  t\right)  +\theta%
{\displaystyle\int_{0}^{T}}
f\left(  s\right)  ds\right]  +\frac{\theta}{2}\left\langle d\Lambda^{\theta
},d\Lambda^{\theta}\right\rangle _{t}+%
{\displaystyle\int_{\Gamma}}
\left\{  \exp\left(  \theta r\left(  t,\lambda\right)  \right)  -1-\theta
r\left(  t,\lambda\right)  \right\}  m\left(  d\lambda\right)  dt\\
+%
{\displaystyle\int_{\Gamma}}
\left\{  \exp\left(  \theta r\left(  t,\lambda\right)  \right)  -1\right\}
\tilde{N}\left(  ds,d\lambda\right)  \\
=\theta Z\left(  t\right)  \exp\left\{  \theta\Lambda^{\theta}\left(
t\right)  +\theta%
{\displaystyle\int_{0}^{T}}
f\left(  s\right)  ds\right\}  dW_{t}+\theta%
{\displaystyle\int_{\Gamma}}
r\left(  t,\lambda\right)  \exp\left\{  \theta\Lambda^{\theta}\left(
t\right)  +\theta%
{\displaystyle\int_{0}^{T}}
f\left(  s\right)  ds\right\}  \tilde{N}\left(  ds,d\lambda\right)  .
\end{array}
\label{I}%
\end{equation}

Hence,
\[%
\begin{array}
[c]{ll}%
\left\langle d\Lambda^{\theta},d\Lambda_{t}^{\theta}\right\rangle  &
=\theta^{2}\left[  Z\left(  t\right)  \exp\left\{  \theta\Lambda^{\theta
}\left(  t\right)  +\theta%
{\displaystyle\int_{0}^{T}}
f\left(  s\right)  ds\right\}  \right]  ^{2}dt\\
& +\theta^{2}%
{\displaystyle\int_{\Gamma}}
\left[  r\left(  t,\lambda\right)  \exp\left\{  \theta\Lambda^{\theta}\left(
t\right)  +\theta%
{\displaystyle\int_{0}^{T}}
f\left(  s\right)  ds\right\}  \right]  ^{2}m(d\lambda)\\
& :=\theta^{2}\left\vert l\left(  t\right)  \right\vert ^{2}dt+\theta^{2}%
{\displaystyle\int_{\Gamma}}
\left\vert L\left(  t,\lambda\right)  \right\vert ^{2}m(d\lambda)dt.
\end{array}
\]

Then, by replacing in $\left(  \ref{I}\right)  ,$ we have the backward
quadratic as the following expression
\[
\left\{
\begin{array}
[c]{ll}%
d\Lambda^{\theta}\left(  t\right)  = & -\left\{  f\left(  t\right)
+\frac{\theta}{2}\left\vert l\left(  t\right)  \right\vert ^{2}+\frac{\theta
}{2}%
{\displaystyle\int_{\Gamma}}
\left\vert L\left(  t,\lambda\right)  \right\vert ^{2}m(d\lambda)\right.  \\
& \left.  +%
{\displaystyle\int_{\Gamma}}
\left(  \dfrac{\exp\left(  \theta r\left(  t,\lambda\right)  \right)
-1}{\theta}-r\left(  t,\lambda\right)  \right)  m(d\lambda)\right\}
dt+l\left(  t\right)  dW\left(  t\right)  \\
& -%
{\displaystyle\int_{\Gamma}}
\left\{  \dfrac{\exp\left(  \theta r\left(  t,\lambda\right)  \right)
-1}{\theta}\right\}  \tilde{N}\left(  dt,d\lambda\right)  +%
{\displaystyle\int_{\Gamma}}
L\left(  t,\lambda\right)  \tilde{N}\left(  dt,d\lambda\right)  ,\\
\Lambda^{\theta}\left(  T\right)  = & \Phi_{x}\left(  x^{u}\left(  T\right)
\right)  +\Psi\left(  y^{u}\left(  0\right)  \right)  ,
\end{array}
\right.
\]

where,%
\[%
\begin{array}
[c]{l}%
l\left(  t\right)  =:Z\left(  t\right)  \exp\left(  \theta\Lambda^{\theta
}\left(  t\right)  +\theta%
{\displaystyle\int_{0}^{T}}
f\left(  s\right)  ds\right)  \\
L\left(  t,.\right)  =:r\left(  t,.\right)  \exp\left(  \theta\Lambda^{\theta
}\left(  t\right)  +\theta%
{\displaystyle\int_{0}^{T}}
f\left(  s\right)  ds\right)  .
\end{array}
\]

\end{proof}

As is proved in lemma $\ref{NSCEEU},$ the process $\Lambda^{\theta}$ is the
first component of the $\mathcal{F}_{t}-$adapted pair of processes $\left(
\Lambda^{\theta},l,L(.)\right)  $ which is the unique solution to the
quadratic backward SDE with jump diffusion $\left(  \ref{EEU}\right)  $.

\begin{lemma}
\label{transform}Suppose that $\left(  \mathbf{H}_{4}\right)  $ holds. Then
\begin{equation}
\mathbb{E}\left(  \sup_{0\leq t\leq T.}\left\vert \Lambda^{\theta}\left(
t\right)  \right\vert ^{2}\right)  \leq C_{T},\label{E<C}%
\end{equation}
In particular, $V^{\theta}$ solves the following linear backward SDE%
\begin{equation}
dV^{\theta}\left(  t\right)  =\theta l\left(  t\right)  V^{\theta}\left(
t\right)  dW\left(  t\right)  +\theta V^{\theta}\left(  t\right)
{\displaystyle\int_{\Gamma}}
L\left(  t,\lambda\right)  \tilde{N}\left(  dt,d\lambda\right)  ,\text{
}V^{\theta}\left(  T\right)  =A_{T}^{\theta}.\label{backward of V}%
\end{equation}
Hence, the process defined on $\left(  \Omega,\mathcal{F},\left(
\mathcal{F}_{t}^{\left(  W,N\right)  }\right)  _{t\geq0},\mathbb{P}\right)  $
by%
\begin{equation}%
\begin{array}
[c]{ll}%
L_{t}^{\theta}:=\frac{V^{\theta}\left(  t\right)  }{V^{\theta}\left(
0\right)  }= & \exp\left(
{\displaystyle\int_{0}^{t}}
\theta l\left(  s\right)  dW\left(  s\right)  -\frac{\theta^{2}}{2}%
{\displaystyle\int_{0}^{t}}
\left\vert l\left(  s\right)  \right\vert ^{2}ds+%
{\displaystyle\int_{0}^{t}}
{\displaystyle\int_{\Gamma}}
L\left(  s,\lambda\right)  \tilde{N}\left(  ds,d\lambda\right)  \right.  \\
& \left.  -%
{\displaystyle\int_{\Gamma}}
\left\{  \frac{\exp\left(  \theta r\left(  t,\lambda\right)  \right)
-1}{\theta}\right\}  \tilde{N}\left(  dt,d\lambda\right)  -\frac{\theta^{2}%
}{2}%
{\displaystyle\int_{0}^{t}}
{\displaystyle\int_{\Gamma}}
\left\vert L\left(  s,\lambda\right)  \right\vert ^{2}m(d\lambda)ds\right.  \\
& \left.  -%
{\displaystyle\int_{\Gamma}}
\left(  \frac{\exp\left(  \theta r\left(  t,\lambda\right)  \right)
-1}{\theta}-r\left(  t,\lambda\right)  \right)  m(d\lambda)\right)  ,\text{
\ \ \ \ \ \ \ \ \ \ \ \ \ \ \ \ \ \ \ }0\leq t\leq T,
\end{array}
\label{exp of V}%
\end{equation}
is a uniformly bounded $\mathcal{F}-$martingale.
\end{lemma}

\begin{proof}
First we prove $\left(  \ref{E<C}\right)  .$ We assume that $(\mathbf{H}%
_{\mathbf{4}})$ holds$,$ $f,$ $\Phi$ and $\Psi$ are bounded by a constant
$C>0,$ we have%
\begin{equation}
0<e^{-\left(  2+T\right)  C\theta}\leq A_{T}^{\theta}\leq e^{\left(
2+T\right)  C\theta}.\label{bounded of AT}%
\end{equation}
Therefore, $V^{\theta}$ is a uniformly bounded $\mathcal{F}_{t}-$martingale
satisfying%
\begin{equation}
0<e^{-\left(  2+T\right)  C\theta}\leq V^{\theta}\left(  t\right)  \leq
e^{\left(  2+T\right)  C\theta},\text{\ }0\leq t\leq T.\label{bounded of V}%
\end{equation}
The complete proof see the Lemma 3.1 page 405 $\cite{Chala 02}$.
\end{proof}

In the next, we will state and prove the necessary optimality conditions for
the system driven by fully coupled FBSDE with jumps diffusion with a risk
sensitive performance functional type. To this end, let us summarize and prove
some lemmas that we will use thereafter.

\begin{lemma}
\label{Lemma p2 p3}The second and the third risk-sensitive adjoint equations
of the solution\newline$\left(  \widetilde{p}_{2}\left(  t\right)
,\widetilde{q}_{2}\left(  t\right)  ,\widetilde{\pi}_{2}\left(  t,\lambda
\right)  \right)  ,$ $\left(  \widetilde{p}_{3}\left(  t\right)
,\widetilde{q}_{3}\left(  t\right)  ,\widetilde{\pi}_{3}\left(  t,.\right)
\right)  $ and $\left(  V^{\theta}\left(  t\right)  ,l\left(  t\right)
,L\left(  t,.\right)  \right)  $ become%
\begin{equation}
\left\{
\begin{array}
[c]{l}%
d\widetilde{p}_{2}\left(  t\right)  =-H_{x}^{\theta}\left(  t\right)
dt+\left(  \widetilde{q}_{2}\left(  t\right)  -\theta l\left(  t\right)
\widetilde{p}_{2}\left(  t\right)  \right)  dW_{t}^{\theta}+%
{\displaystyle\int_{\Gamma}}
\left(  \widetilde{\pi}_{2}\left(  t,\lambda\right)  -\theta L\left(
t,\lambda\right)  \widetilde{p}_{2}\left(  t\right)  \right)  \tilde
{N}^{\theta}\left(  dt,d\lambda\right)  ,\\
d\widetilde{p}_{3}\left(  t\right)  =-H_{y}^{\theta}\left(  t\right)
dt-\left(  H_{z}^{\theta}\left(  t\right)  -\theta l\left(  t\right)
\widetilde{p}_{3}\left(  t\right)  \right)  dW_{t}^{\theta}-%
{\displaystyle\int_{\Gamma}}
\left(  \nabla H_{r}\left(  t\right)  -\theta L\left(  t,\lambda\right)
\widetilde{p}_{3}\left(  t\right)  \right)  \tilde{N}^{\theta}\left(
dt,d\lambda\right)  ,\\
dV^{\theta}\left(  t\right)  =\theta V^{\theta}\left(  t\right)  l\left(
t\right)  dW_{t}+\theta V^{\theta}\left(  t\right)  \int_{\Gamma}L\left(
t,\lambda\right)  \tilde{N}\left(  dt,d\lambda\right)  ,\\
V^{\theta}\left(  T\right)  =A^{\theta}\left(  T\right)  ,\\
\widetilde{p}_{2}\left(  T\right)  =\Phi_{x}\left(  x_{T}\right)  ,\text{
}\widetilde{p}_{3}\left(  0\right)  =\Psi_{y}\left(  y\left(  0\right)
\right)  .
\end{array}
\right.  \label{system transformed1}%
\end{equation}
The solution $\left(  \widetilde{p}\left(  t\right)  ,\widetilde{q}\left(
t\right)  ,\widetilde{\pi}\left(  t,.\right)  ,V^{\theta}\left(  t\right)
,l\left(  t\right)  ,L\left(  t,.\right)  \right)  $ of the system $\left(
\ref{system transformed1}\right)  $ is unique, such that%
\begin{align}
&  \left.  \mathbb{E}\left[  \sup_{0\leq t\leq T}\left\vert \widetilde
{p}\left(  t\right)  \right\vert ^{2}+\sup_{0\leq t\leq T}\left\vert
V^{\theta}\left(  t\right)  \right\vert ^{2}+%
{\displaystyle\int_{0}^{T}}
\left(  \left\vert \widetilde{q}\left(  t\right)  \right\vert ^{2}+\left\vert
l\left(  t\right)  \right\vert ^{2}\right.  \right.  \right.
\label{condition limit1}\\
&  \left.  \left.  \left.  +%
{\displaystyle\int_{\Gamma}}
\left(  \left\vert \widetilde{\pi}\left(  t,\lambda\right)  \right\vert
^{2}+\left\vert L\left(  t,\lambda\right)  \right\vert ^{2}\right)  m\left(
d\lambda\right)  \right)  dt\right]  <\infty,\right. \nonumber
\end{align}
where
\begin{equation}%
\begin{array}
[c]{l}%
H^{\theta}\left(  t,x\left(  t\right)  ,y\left(  t\right)  ,z\left(  t\right)
,r(t,.),\widetilde{p}\left(  t\right)  ,\widetilde{q}\left(  t\right)
,\widetilde{\pi}\left(  t,\lambda\right)  ,V^{\theta}\left(  t\right)
,l\left(  t\right)  ,L\left(  t,.\right)  \right) \\
=f\left(  t\right)  +b\left(  t\right)  \widetilde{p}_{2}+\sigma\left(
t\right)  \widetilde{q}_{2}+\left(  g\left(  t\right)  -\theta z\left(
t\right)  l\left(  t\right)  \right)  \widetilde{p}_{3}\\
+%
{\displaystyle\int_{\Gamma}}
\left\{  \gamma\left(  t^{-},\lambda\right)  \widetilde{\pi}_{2}\left(
t,\lambda\right)  -\left(  g\left(  t\right)  -\theta r\left(  t,\lambda
\right)  L\left(  t,\lambda\right)  \right)  \widetilde{p}_{3}\right\}
\lambda m\left(  d\lambda\right)  .
\end{array}
\label{H risk-sensitive1}%
\end{equation}

\end{lemma}

\begin{proof}
We want to identify the processes $\widetilde{\alpha},\widetilde{\beta}$ and
$\widetilde{\gamma}$ such that%
\[
d\widetilde{p}\left(  t\right)  =-\widetilde{\alpha}\left(  t\right)
dt+\widetilde{\beta}\left(  t\right)  dW\left(  t\right)  +%
{\displaystyle\int_{\Gamma}}
\widetilde{\gamma}\left(  t^{-},\lambda\right)  \widetilde{N}\left(
d\lambda,dt\right)
\]
By applying It\^{o}'s formula to the process $\overrightarrow{p}\left(
t\right)  =\theta V^{\theta}\left(  t\right)  \widetilde{p}\left(  t\right)
,$ and using the expression of $V^{\theta}$ in $\left(  \ref{backward of V}%
\right)  ,$ we obtain%
\[%
\begin{array}
[c]{ll}%
d\widetilde{p}\left(  t\right)  = & -\left[  \frac{1}{\theta V^{\theta}\left(
t\right)  }\left(
\begin{tabular}
[c]{lll}%
$0$ & $0$ & $0$\\
$f_{x}\left(  t\right)  $ & $b_{x}\left(  t\right)  $ & $g_{x}\left(
t\right)  $\\
$f_{y}\left(  t\right)  $ & $b_{y}\left(  t\right)  $ & $g_{y}\left(
t\right)  $%
\end{tabular}
\right)  \left(
\begin{tabular}
[c]{l}%
$p_{1}\left(  t\right)  $\\
$p_{2}\left(  t\right)  $\\
$p_{3}\left(  t\right)  $%
\end{tabular}
\right)  \right. \\
& \left.  +\frac{1}{\theta V^{\theta}\left(  t\right)  }\left(
\begin{tabular}
[c]{lll}%
$0$ & $0$ & $0$\\
$0$ & $\sigma_{x}\left(  t\right)  $ & $0$\\
$0$ & $\sigma_{y}\left(  t\right)  $ & $0$%
\end{tabular}
\right)  \left(
\begin{tabular}
[c]{l}%
$q_{1}\left(  t\right)  $\\
$q_{2}\left(  t\right)  $\\
$q_{3}\left(  t\right)  $%
\end{tabular}
\right)  -\theta\left(
\begin{tabular}
[c]{l}%
$l_{1}\left(  t\right)  $\\
$l_{2}\left(  t\right)  $\\
$l_{3}\left(  t\right)  $%
\end{tabular}
\right)  \widetilde{\beta}\left(  t\right)  \right. \\
& \left.  -\frac{1}{\theta V^{\theta}\left(  t\right)  }%
{\displaystyle\int_{\Gamma}}
\left(  \left(
\begin{tabular}
[c]{lll}%
$0$ & $0$ & $0$\\
$0$ & $\gamma_{x}\left(  t-,\lambda\right)  $ & $0$\\
$0$ & $\gamma_{y}\left(  t-,\lambda\right)  $ & $0$%
\end{tabular}
\right)  \left(
\begin{tabular}
[c]{l}%
$\pi_{1}\left(  t,\lambda\right)  $\\
$\pi_{2}\left(  t,\lambda\right)  $\\
$\pi_{3}\left(  t,\lambda\right)  $%
\end{tabular}
\right)  \right.  \right. \\
& \left.  \left.  -\theta%
{\displaystyle\int_{\Gamma}}
\left(
\begin{tabular}
[c]{l}%
$L_{1}\left(  t,\lambda\right)  $\\
$L_{2}\left(  t,\lambda\right)  $\\
$L_{3}\left(  t,\lambda\right)  $%
\end{tabular}
\right)  \widetilde{\gamma}\left(  t\right)  \right)  m(d\lambda)\right]  dt\\
& +\left[  \frac{1}{\theta V^{\theta}\left(  t\right)  }\left(
\begin{tabular}
[c]{l}%
$q_{1}\left(  t\right)  $\\
$q_{2}\left(  t\right)  $\\
$q_{3}\left(  t\right)  $%
\end{tabular}
\right)  -\theta\left(
\begin{tabular}
[c]{l}%
$l_{1}\left(  t\right)  $\\
$l_{2}\left(  t\right)  $\\
$l_{3}\left(  t\right)  $%
\end{tabular}
\right)  \widetilde{p}\left(  t\right)  \right]  dW\left(  t\right) \\
& +\left[  \frac{1}{\theta V^{\theta}\left(  t\right)  }%
{\displaystyle\int_{\Gamma}}
\left(
\begin{tabular}
[c]{l}%
$\pi_{1}\left(  t,\lambda\right)  $\\
$\pi_{2}\left(  t,\lambda\right)  $\\
$\pi_{3}\left(  t,\lambda\right)  $%
\end{tabular}
\right)  -\theta%
{\displaystyle\int_{\Gamma}}
\left(
\begin{tabular}
[c]{l}%
$L_{1}\left(  t,\lambda\right)  $\\
$L_{2}\left(  t,\lambda\right)  $\\
$L_{3}\left(  t,\lambda\right)  $%
\end{tabular}
\right)  \widetilde{p}\left(  t\right)  \right]  \widetilde{N}\left(
d\lambda,dt\right)
\end{array}
\]

By identifying the coefficients, and using the relation $\widetilde{p}\left(
t\right)  =\dfrac{1}{\theta V^{\theta}\left(  t\right)  }\overrightarrow
{p}\left(  t\right)  ,$ the diffusion coefficient $\widetilde{q}\left(
t\right)  $ will be%
\[
\widetilde{\beta}\left(  t\right)  =\left(
\begin{tabular}
[c]{l}%
$\widetilde{q}_{1}\left(  t\right)  $\\
$\widetilde{q}_{2}\left(  t\right)  $\\
$\widetilde{q}_{3}\left(  t\right)  $%
\end{tabular}
\ \ \ \ \ \right)  -\theta\left(
\begin{tabular}
[c]{l}%
$l_{1}\left(  t\right)  $\\
$l_{2}\left(  t\right)  $\\
$l_{3}\left(  t\right)  $%
\end{tabular}
\ \ \ \ \ \right)  \widetilde{p}\left(  t\right)  ,
\]
the drift term of the process $\widetilde{p}\left(  t\right)  $%
\[%
\begin{array}
[c]{ll}%
\widetilde{\alpha}\left(  t\right)  = & \left(
\begin{tabular}
[c]{lll}%
$0$ & $0$ & $0$\\
$f_{x}\left(  t\right)  $ & $b_{x}\left(  t\right)  $ & $g_{x}\left(
t\right)  $\\
$f_{y}\left(  t\right)  $ & $b_{y}\left(  t\right)  $ & $g_{y}\left(
t\right)  $%
\end{tabular}
\ \ \ \ \ \right)  \left(
\begin{tabular}
[c]{l}%
$\widetilde{p}_{1}\left(  t\right)  $\\
$\widetilde{p}_{2}\left(  t\right)  $\\
$\widetilde{p}_{3}\left(  t\right)  $%
\end{tabular}
\ \ \ \ \ \right)  \\
& +\left(
\begin{tabular}
[c]{lll}%
$0$ & $0$ & $0$\\
$0$ & $\sigma_{x}\left(  t\right)  $ & $0$\\
$0$ & $\sigma_{y}\left(  t\right)  $ & $0$%
\end{tabular}
\ \ \ \ \ \right)  \left(
\begin{tabular}
[c]{l}%
$\widetilde{q}_{1}\left(  t\right)  $\\
$\widetilde{q}_{2}\left(  t\right)  $\\
$\widetilde{q}_{3}\left(  t\right)  $%
\end{tabular}
\ \ \ \ \ \right)  \\
& +\theta\left(
\begin{tabular}
[c]{l}%
$l_{1}\left(  t\right)  $\\
$l_{2}\left(  t\right)  $\\
$l_{3}\left(  t\right)  $%
\end{tabular}
\ \ \ \ \ \ \right)  \widetilde{\beta}\left(  t\right)  -%
{\displaystyle\int_{\Gamma}}
\left(  \left(
\begin{tabular}
[c]{lll}%
$0$ & $0$ & $0$\\
$0$ & $\gamma_{x}\left(  t-,\lambda\right)  $ & $0$\\
$0$ & $\gamma_{y}\left(  t-,\lambda\right)  $ & $0$%
\end{tabular}
\ \ \ \ \ \right)  \right.  \\
& \left.  \left(
\begin{tabular}
[c]{l}%
$\pi_{1}\left(  t,\lambda\right)  $\\
$\pi_{2}\left(  t,\lambda\right)  $\\
$\pi_{3}\left(  t,\lambda\right)  $%
\end{tabular}
\ \ \ \ \ \right)  -\theta\left(
\begin{tabular}
[c]{l}%
$L_{1}\left(  t,\lambda\right)  $\\
$L_{2}\left(  t,\lambda\right)  $\\
$L_{3}\left(  t,\lambda\right)  $%
\end{tabular}
\ \ \ \ \ \right)  \widetilde{\gamma}\left(  t^{-},\lambda\right)  \right)
m\left(  d\lambda\right)  .
\end{array}
\]
the jump diffusion gets the form%
\[%
\begin{array}
[c]{ll}%
\widetilde{\gamma}\left(  t^{-},.\right)  = & \left(
\begin{tabular}
[c]{l}%
$\widetilde{\pi}_{1}\left(  t,.\right)  $\\
$\widetilde{\pi}_{2}\left(  t,.\right)  $\\
$\widetilde{\pi}_{3}\left(  t,.\right)  $%
\end{tabular}
\ \ \ \ \ \right)  -\theta\left(
\begin{tabular}
[c]{l}%
$L_{1}\left(  t,.\right)  $\\
$L_{2}\left(  t,.\right)  $\\
$L_{3}\left(  t,.\right)  $%
\end{tabular}
\ \ \ \ \ \right)  \widetilde{p}\left(  t\right)
\end{array}
\]
Finally, we obtain%
\[%
\begin{array}
[c]{ll}%
d\widetilde{p}\left(  t\right)  = & -\left[  \left(
\begin{tabular}
[c]{lll}%
$0$ & $0$ & $0$\\
$f_{x}\left(  t\right)  $ & $b_{x}\left(  t\right)  $ & $g_{x}\left(
t\right)  $\\
$f_{y}\left(  t\right)  $ & $b_{y}\left(  t\right)  $ & $g_{y}\left(
t\right)  $%
\end{tabular}
\ \ \ \ \ \ \right)  \left(
\begin{tabular}
[c]{l}%
$\widetilde{p}_{1}\left(  t\right)  $\\
$\widetilde{p}_{2}\left(  t\right)  $\\
$\widetilde{p}_{3}\left(  t\right)  $%
\end{tabular}
\ \ \ \ \ \ \right)  \right.  \\
& \left.  +\left(
\begin{tabular}
[c]{lll}%
$0$ & $0$ & $0$\\
$0$ & $\sigma_{x}\left(  t\right)  $ & $0$\\
$0$ & $\sigma_{y}\left(  t\right)  $ & $0$%
\end{tabular}
\ \ \ \ \ \ \right)  \left(
\begin{tabular}
[c]{l}%
$\widetilde{q}_{1}\left(  t\right)  $\\
$\widetilde{q}_{2}\left(  t\right)  $\\
$\widetilde{q}_{3}\left(  t\right)  $%
\end{tabular}
\ \ \ \ \ \right)  -\theta\left(
\begin{tabular}
[c]{l}%
$l_{1}\left(  t\right)  $\\
$l_{2}\left(  t\right)  $\\
$l_{3}\left(  t\right)  $%
\end{tabular}
\ \ \ \ \ \ \right)  \widetilde{\beta}\left(  t\right)  \right.  \\
& \left.  -%
{\displaystyle\int_{\Gamma}}
\left(  \left(
\begin{tabular}
[c]{lll}%
$0$ & $0$ & $0$\\
$0$ & $\gamma_{x}\left(  t-,\lambda\right)  $ & $0$\\
$0$ & $\gamma_{y}\left(  t-,\lambda\right)  $ & $0$%
\end{tabular}
\ \ \ \ \ \ \right)  \left(
\begin{tabular}
[c]{l}%
$\pi_{1}\left(  t,\lambda\right)  $\\
$\pi_{2}\left(  t,\lambda\right)  $\\
$\pi_{3}\left(  t,\lambda\right)  $%
\end{tabular}
\ \ \ \ \ \right)  \right.  \right.  \\
& \left.  \left.  -\theta\left(
\begin{tabular}
[c]{l}%
$L_{1}\left(  t,\lambda\right)  $\\
$L_{2}\left(  t,\lambda\right)  $\\
$L_{3}\left(  t,\lambda\right)  $%
\end{tabular}
\ \ \ \ \ \right)  \widetilde{\gamma}\left(  t,\lambda\right)  \right)
m\left(  d\lambda\right)  \right]  dt\\
& +\left[  \left(
\begin{tabular}
[c]{l}%
$\widetilde{q}_{1}\left(  t\right)  $\\
$\widetilde{q}_{2}\left(  t\right)  $\\
$\widetilde{q}_{3}\left(  t\right)  $%
\end{tabular}
\ \ \ \ \ \right)  -\theta\left(
\begin{tabular}
[c]{l}%
$l_{1}\left(  t\right)  $\\
$l_{2}\left(  t\right)  $\\
$l_{3}\left(  t\right)  $%
\end{tabular}
\ \ \ \ \ \right)  \widetilde{p}\left(  t\right)  \right]  dW\left(  t\right)
\\
& +%
{\displaystyle\int_{\Gamma}}
\left[  \left(
\begin{tabular}
[c]{l}%
$\widetilde{\pi}_{1}\left(  t,\lambda\right)  $\\
$\widetilde{\pi}_{2}\left(  t,\lambda\right)  $\\
$\widetilde{\pi}_{3}\left(  t,\lambda\right)  $%
\end{tabular}
\ \ \ \ \ \right)  -\theta\left(
\begin{tabular}
[c]{l}%
$L_{1}\left(  t,\lambda\right)  $\\
$L_{2}\left(  t,\lambda\right)  $\\
$L_{3}\left(  t,\lambda\right)  $%
\end{tabular}
\ \ \ \ \ \right)  \widetilde{p}\left(  t\right)  \right]  \widetilde
{N}\left(  d\lambda,dt\right)  .
\end{array}
\]
It is easily verified that%
\[
\left\{
\begin{array}
[c]{ll}%
d\widetilde{p}_{1}\left(  t\right)  = & \widetilde{q}_{1}\left(  t\right)
\left[  -\theta l_{1}\left(  t\right)  dt+dW\left(  t\right)  \right]  +%
{\displaystyle\int_{\Gamma}}
\widetilde{\pi}_{1}\left(  t,\lambda\right)  \left[  -\theta L_{1}\left(
t,\lambda\right)  m(d\lambda)dt+\widetilde{N}\left(  d\lambda,dt\right)
\right]  \\
\widetilde{p}_{1}\left(  T\right)  = & 1
\end{array}
\right.  .
\]
In view of $\left(  \ref{exp of V}\right)  ,$ we may use Girsanov's Theorem to
claim that%
\[
\left\{
\begin{array}
[c]{ll}%
d\widetilde{p}_{1}\left(  t\right)  = & \widetilde{q}_{1}\left(  t\right)
dW^{\theta}\left(  t\right)  +%
{\displaystyle\int_{\Gamma}}
\widetilde{\pi}_{1}\left(  t,\lambda\right)  \widetilde{N}^{\theta}\left(
d\lambda,dt\right)  \\
\widetilde{p}_{1}\left(  T\right)  = & 1
\end{array}
\right.  ,\text{ \ \ \ }\mathbb{P}^{\theta}-as,
\]
where,
\begin{equation}%
\begin{array}
[c]{l}%
dW^{\theta}\left(  t\right)  =-\theta l\left(  t\right)  dt+dW\left(
t\right)  \\
\widetilde{N}^{\theta}\left(  d\lambda,dt\right)  =-\theta L\left(
t,\lambda\right)  m\left(  d\lambda\right)  +\widetilde{N}\left(
d\lambda,dt\right)  ,
\end{array}
\label{Girsanov trans}%
\end{equation}
$W^{\theta}\left(  t\right)  $ is a $\mathbb{P}^{\theta}-$Brownian motion and
$\widetilde{N}^{\theta}\left(  \lambda,t\right)  $ is a $\mathbb{P}^{\theta}%
-$compensator Poisson measure, where,
\[%
\begin{array}
[c]{ll}%
\left.  \dfrac{d\mathbb{P}^{\theta}}{d\mathbb{P}}\right\vert _{\mathcal{F}%
_{t}}:= & L_{t}^{\theta}=\exp\left(
{\displaystyle\int_{0}^{t}}
\theta l\left(  s\right)  dW\left(  s\right)  -\dfrac{\theta^{2}}{2}%
{\displaystyle\int_{0}^{t}}
\left\vert l\left(  s\right)  \right\vert ^{2}ds+%
{\displaystyle\int_{0}^{t}}
{\displaystyle\int_{\Gamma}}
L\left(  s,\lambda\right)  \tilde{N}\left(  ds,d\lambda\right)  \right.  \\
& \left.  -%
{\displaystyle\int_{\Gamma}}
\left\{  \dfrac{\exp\left(  \theta r\left(  t,\lambda\right)  \right)
-1}{\theta}\right\}  \tilde{N}\left(  dt,d\lambda\right)  -\dfrac{\theta^{2}%
}{2}%
{\displaystyle\int_{0}^{t}}
{\displaystyle\int_{\Gamma}}
\left\vert L\left(  s,\lambda\right)  \right\vert ^{2}m(d\lambda)ds\right.  \\
& \left.  -%
{\displaystyle\int_{\Gamma}}
\left(  \dfrac{\exp\left(  \theta r\left(  t,\lambda\right)  \right)
-1}{\theta}-r\left(  t,\lambda\right)  \right)  m(d\lambda)\right)
\ \ \ \ \ \ 0\leq t\leq T.
\end{array}
\]
\ But according to $\left(  \ref{exp of V}\right)  $ and $\left(
\ref{bounded of AT}\right)  ,$ the probability measures $\mathbb{P}^{\theta}$
and $\mathbb{P}$ are in fact equivalent. Hence, noting that $\widetilde{p}%
_{1}\left(  t\right)  :=\dfrac{1}{\theta V^{\theta}\left(  t\right)  }%
p_{1}\left(  t\right)  $ is square-integrable, we get that
\[
\widetilde{p}_{1}\left(  t\right)  =\mathbb{E}^{\mathbb{P}^{\theta}}\left[
\widetilde{p}_{1}\left(  T\right)  \mid\mathcal{F}_{t}\right]  =1.
\]
Thus, its quadratic variation $%
{\displaystyle\int_{0}^{T}}
\left\vert \widetilde{q}_{1}\left(  t\right)  \right\vert ^{2}dt=0.$ This
implies that, for almost every $0\leq t\leq T,$ $\widetilde{q}_{1}\left(
t\right)  =0,$ \ $\mathbb{P}^{\theta}$ and $\mathbb{P-}$a.s. Now we use the
relations
\[
\widetilde{q}\left(  t\right)  =\left(
\begin{tabular}
[c]{l}%
$\widetilde{q}_{1}\left(  t\right)  $\\
$\widetilde{q}_{2}\left(  t\right)  $\\
$-\widetilde{H}_{z}\left(  t\right)  $%
\end{tabular}
\ \ \ \ \ \right)  -\theta\left(
\begin{tabular}
[c]{l}%
$l_{1}\left(  t\right)  $\\
$l_{2}\left(  t\right)  $\\
$l_{3}\left(  t\right)  $%
\end{tabular}
\ \ \ \ \ \right)  \widetilde{p}\left(  t\right)  ,
\]
and%
\[
\widetilde{\pi}\left(  t,.\right)  =\left(
\begin{tabular}
[c]{l}%
$\widetilde{\pi}_{1}\left(  t,.\right)  $\\
$\widetilde{\pi}_{2}\left(  t,.\right)  $\\
$-\nabla_{r}\widetilde{H}\left(  t\right)  $%
\end{tabular}
\ \ \ \ \ \right)  -\theta\left(
\begin{tabular}
[c]{l}%
$L_{1}\left(  t,.\right)  $\\
$L_{2}\left(  t,.\right)  $\\
$L_{3}\left(  t,.\right)  $%
\end{tabular}
\ \ \ \ \ \right)  \widetilde{p}\left(  t\right)  ,
\]

in the equation above, to obtain%
\begin{equation}%
\begin{array}
[c]{ll}%
d\widetilde{p}\left(  t\right)  = & -\left\{  \left(
\begin{tabular}
[c]{lll}%
$0$ & $0$ & $0$\\
$f_{x}\left(  t\right)  $ & $b_{x}\left(  t\right)  $ & $g_{x}\left(
t\right)  $\\
$f_{y}\left(  t\right)  $ & $b_{y}\left(  t\right)  $ & $g_{y}\left(
t\right)  $%
\end{tabular}
\ \ \ \right)  \left(
\begin{tabular}
[c]{l}%
$\widetilde{p}_{1}\left(  t\right)  $\\
$\widetilde{p}_{2}\left(  t\right)  $\\
$\widetilde{p}_{3}\left(  t\right)  $%
\end{tabular}
\ \ \ \right)  \right. \\
& +\left.  \left(
\begin{tabular}
[c]{lll}%
$0$ & $0$ & $0$\\
$0$ & $\sigma_{x}\left(  t\right)  $ & $0$\\
$0$ & $\sigma_{y}\left(  t\right)  $ & $0$%
\end{tabular}
\ \ \ \right)  \left(
\begin{tabular}
[c]{l}%
$\widetilde{q}_{1}\left(  t\right)  $\\
$\widetilde{q}_{2}\left(  t\right)  $\\
$\widetilde{q}_{3}\left(  t\right)  $%
\end{tabular}
\ \ \ \right)  \right. \\
& +\left.
{\displaystyle\int_{\Gamma}}
\left(
\begin{tabular}
[c]{lll}%
$0$ & $0$ & $0$\\
$0$ & $\gamma_{x}\left(  t-,\lambda\right)  $ & $0$\\
$0$ & $\gamma_{y}\left(  t-,\lambda\right)  $ & $0$%
\end{tabular}
\ \ \ \right)  \left(
\begin{tabular}
[c]{l}%
$\pi_{1}\left(  t,\lambda\right)  $\\
$\pi_{2}\left(  t,\lambda\right)  $\\
$\pi_{3}\left(  t,\lambda\right)  $%
\end{tabular}
\ \ \ \right)  m\left(  d\lambda\right)  \right\}  dt\\
& +\left\{  \left(
\begin{array}
[c]{c}%
\widetilde{q}_{1}\left(  t\right) \\
\widetilde{q}_{2}\left(  t\right) \\
-f_{z}\left(  t\right)  \widetilde{p}_{1}-b_{z}\left(  t\right)  \widetilde
{p}_{2}-g_{z}\left(  t\right)  \widetilde{p}_{3}-\sigma_{z}\left(  t\right)
\widetilde{q}_{2}%
\end{array}
\right.  \right. \\
& \left.  \left.
\begin{tabular}
[c]{l}%
\\
\\
$+%
{\displaystyle\int_{\Gamma}}
\gamma_{z}\left(  t-,\lambda\right)  \widetilde{\pi}_{2}\left(  t,\lambda
\right)  m\left(  d\lambda\right)  $%
\end{tabular}
\ \ \ \right)  +\theta\left(
\begin{tabular}
[c]{l}%
$l_{1}\left(  t\right)  $\\
$l_{2}\left(  t\right)  $\\
$l_{3}\left(  t\right)  $%
\end{tabular}
\ \ \ \right)  \widetilde{p}\left(  t\right)  \right\}  dW^{\theta}\left(
t\right) \\
& +%
{\displaystyle\int_{\Gamma}}
\left\{  \left(
\begin{tabular}
[c]{l}%
$\widetilde{\pi}_{1}\left(  t,\lambda\right)  $\\
$\widetilde{\pi}_{2}\left(  t,\lambda\right)  $\\
$-f_{r}\left(  t\right)  \widetilde{p}_{1}-b_{r}\left(  t\right)
\widetilde{p}_{2}-g_{r}\left(  t\right)  \widetilde{p}_{3}-\sigma_{r}\left(
t\right)  \widetilde{q}_{2}$%
\end{tabular}
\ \ \ \right.  \right. \\
& \left.  \left.
\begin{tabular}
[c]{l}%
\\
\\
$+%
{\displaystyle\int_{\Gamma}}
\gamma_{r}\left(  t-,\lambda\right)  \widetilde{\pi}_{2}\left(  t,\lambda
\right)  m\left(  d\lambda\right)  $%
\end{tabular}
\ \ \ \right)  -\theta\left(
\begin{tabular}
[c]{l}%
$L_{1}\left(  t,\lambda\right)  $\\
$L_{2}\left(  t,\lambda\right)  $\\
$L_{3}\left(  t,\lambda\right)  $%
\end{tabular}
\ \ \ \right)  \widetilde{p}\left(  t\right)  \right\}  \widetilde{N}^{\theta
}\left(  d\lambda,dt\right)  .
\end{array}
\label{matrix of p}%
\end{equation}
Therefore, the second and third components of $\widetilde{p}_{2}$ and
$\widetilde{p}_{3}$ in $\left(  \ref{matrix of p}\right)  ,$ are given by%
\begin{equation}
\left\{
\begin{array}
[c]{ll}%
d\widetilde{p}_{2}\left(  t\right)  = & -\left\{  f_{x}\left(  t\right)
+b_{x}\left(  t\right)  \widetilde{p}_{2}\left(  t\right)  +g_{x}\left(
t\right)  \widetilde{p}_{3}\left(  t\right)  +\sigma_{x}\left(  t\right)
\widetilde{q}_{2}\left(  t\right)  \right. \\
& \left.  -%
{\displaystyle\int_{\Gamma}}
\gamma_{x}\left(  t-,\lambda\right)  \widetilde{\pi}_{2}\left(  t,\lambda
\right)  m\left(  d\lambda\right)  \right\}  dt\\
& +\left\{  \widetilde{q}_{2}\left(  t\right)  -\theta l_{2}\left(  t\right)
\widetilde{p}_{2}\left(  t\right)  \right\}  dW^{\theta}\left(  t\right)  +%
{\displaystyle\int_{\Gamma}}
\left\{  \widetilde{\pi}_{2}\left(  t,\lambda\right)  -\theta L_{2}\left(
t,\lambda\right)  \widetilde{p}_{2}\left(  t\right)  \right\}  \widetilde
{N}^{\theta}\left(  d\lambda,dt\right)  ,\\
\widetilde{p}_{2}\left(  T\right)  = & \Phi_{x}\left(  x_{T}\right)  ,
\end{array}
\right.  \label{P2tilde}%
\end{equation}

and%
\begin{equation}
\left\{
\begin{array}
[c]{ll}%
d\widetilde{p}_{3}\left(  t\right)  = & -\left\{  f_{y}\left(  t\right)
+b_{y}\left(  t\right)  \widetilde{p}_{2}\left(  t\right)  +g_{y}\left(
t\right)  \widetilde{p}_{3}\left(  t\right)  +\sigma_{y}\left(  t\right)
\widetilde{q}_{2}\left(  t\right)  \right. \\
& \left.  +\theta l_{3}\left(  t\right)  \widetilde{q}_{3}\left(  t\right)  -%
{\displaystyle\int_{\Gamma}}
\gamma_{y}\left(  t-,\lambda\right)  \widetilde{\pi}_{2}\left(  t,\lambda
\right)  m\left(  d\lambda\right)  \right\}  dt\\
& -\left\{  \left\{  f_{z}\left(  t\right)  +b_{z}\left(  t\right)
\widetilde{p}_{2}\left(  t\right)  +g_{z}\left(  t\right)  \widetilde{p}%
_{3}\left(  t\right)  +\sigma_{z}\left(  t\right)  \widetilde{q}_{2}\left(
t\right)  \right\}  \right. \\
& \left.  +%
{\displaystyle\int_{\Gamma}}
\gamma_{z}\left(  t-,\lambda\right)  \widetilde{\pi}_{2}\left(  t,\lambda
\right)  m\left(  d\lambda\right)  +\theta l_{3}\left(  t\right)
\widetilde{p}_{3}\left(  t\right)  \right\}  dW^{\theta}\left(  t\right) \\
& -%
{\textstyle\int_{\Gamma}}
\left\{  f_{r}\left(  t\right)  +b_{r}\left(  t\right)  \widetilde{p}%
_{2}\left(  t\right)  +g_{r}\left(  t\right)  \widetilde{p}_{3}\left(
t\right)  +\sigma_{r}\left(  t\right)  \widetilde{q}_{2}\left(  t\right)
\right. \\
& \left.  -%
{\displaystyle\int_{\Gamma}}
\left(  \gamma_{r}\left(  t-,\lambda\right)  \widetilde{\pi}_{2}\left(
t,\lambda\right)  +\theta L_{3}\left(  t,\lambda\right)  \widetilde{p}%
_{3}\left(  t\right)  \right)  m\left(  d\lambda\right)  \right\}
\widetilde{N}^{\theta}\left(  d\lambda,dt\right)  ,\\
\widetilde{p}_{3}\left(  0\right)  = & \Psi_{y}\left(  y\left(  0\right)
\right)  ,
\end{array}
\right.  \label{P3tilde}%
\end{equation}
or in equivalent expression the adjoint equations for $\left(  \widetilde
{p}_{2},\widetilde{q}_{2}\right)  ,$ $\left(  \widetilde{p}_{3},\widetilde
{q}_{3}\right)  ,$ $\left(  \widetilde{\pi}_{2},\widetilde{\pi}_{3}\right)  $
and $\left(  V^{\theta},l,L\right)  $ become%
\[
\left\{
\begin{array}
[c]{l}%
d\widetilde{p}_{2}\left(  t\right)  =-H_{x}^{\theta}\left(  t\right)
dt+\left(  \widetilde{q}_{2}\left(  t\right)  -\theta l_{2}\left(  t\right)
\widetilde{p}_{2}\right)  dW^{\theta}\left(  t\right)  +%
{\displaystyle\int_{\Gamma}}
\left\{  \widetilde{\pi}_{2}\left(  t,\lambda\right)  -\theta L_{2}\left(
t,\lambda\right)  \widetilde{p}_{2}\left(  t\right)  \right\}  \widetilde
{N}^{\theta}\left(  d\lambda,dt\right)  ,\\
d\widetilde{p}_{3}\left(  t\right)  =-H_{y}^{\theta}\left(  t\right)
dt-H_{z}^{\theta}\left(  t\right)  dW^{\theta}\left(  t\right)  -%
{\displaystyle\int_{\Gamma}}
\nabla H_{r}^{\theta}\left(  t\right)  \widetilde{N}^{\theta}\left(
d\lambda,dt\right)  ,\\
dV^{\theta}\left(  t\right)  =\theta l\left(  t\right)  V^{\theta}\left(
t\right)  dW\left(  t\right)  +\theta V^{\theta}\left(  t\right)
{\displaystyle\int_{\Gamma}}
L\left(  t,\lambda\right)  \widetilde{N}\left(  d\lambda,dt\right)  ,\\
V^{\theta}\left(  T\right)  =A^{\theta}\left(  T\right)  ,\\
\widetilde{p}_{2}\left(  T\right)  =\Phi_{x}\left(  x\left(  T\right)
\right)  ,\text{ }\widetilde{p}_{3}\left(  0\right)  =\Psi_{y}\left(  y\left(
0\right)  \right)  .
\end{array}
\right.
\]
The solution $\left(  \widetilde{p},\widetilde{q},\widetilde{\pi},V^{\theta
},l,L\right)  $ of the system $\left(  \ref{system transformed1}\right)  $ is
unique, such that%
\[%
\begin{array}
[c]{l}%
\mathbb{E}\left[  \sup\limits_{0\leq t\leq T}\left\vert \widetilde{p}\left(
t\right)  \right\vert ^{2}+\sup\limits_{0\leq t\leq T}\left\vert V^{\theta
}\left(  t\right)  \right\vert ^{2}+%
{\displaystyle\int_{0}^{T}}
\left(  \left\vert \widetilde{q}\left(  t\right)  \right\vert ^{2}+\left\vert
l\left(  t\right)  \right\vert ^{2}\right.  \right. \\
\left.  \left.  +%
{\displaystyle\int_{\Gamma}}
\left(  \left\vert \widetilde{\pi}\left(  t,\lambda\right)  \right\vert
^{2}+\left\vert L\left(  t,\lambda\right)  \right\vert ^{2}\right)  m\left(
d\lambda\right)  \right)  dt\right]  <\infty,
\end{array}
\]
where
\[%
\begin{array}
[c]{ll}%
H^{\theta}\left(  t\right)  & :=H^{\theta}\left(  t,x\left(  t\right)
,y\left(  t\right)  ,z\left(  t\right)  ,r^{u}\left(  t,\lambda\right)
,\widetilde{p}_{2}\left(  t\right)  ,\widetilde{q}_{2}\left(  t\right)
,\right. \\
& \left.  \widetilde{p}_{3}\left(  t\right)  ,\widetilde{\pi}_{2}\left(
t,\lambda\right)  ,V^{\theta}\left(  t\right)  ,l\left(  t\right)  ,L\left(
t,\lambda\right)  \right) \\
& =f\left(  t\right)  +b\left(  t\right)  \widetilde{p}_{2}+\sigma\left(
t\right)  \widetilde{q}_{2}+\left(  g\left(  t\right)  +z\left(  t\right)
\theta l\left(  t\right)  \right)  \widetilde{p}_{3}\\
& -%
{\displaystyle\int_{\Gamma}}
\left\{  \gamma\left(  t-,\lambda\right)  \widetilde{\pi}_{2}\left(
t,\lambda\right)  -\left(  g\left(  t\right)  +r\left(  t,\lambda\right)
L\left(  t,\lambda\right)  \right)  \widetilde{p}_{3}\right\}  m\left(
d\lambda\right)  .
\end{array}
\]

The proof is completed.
\end{proof}

\begin{theorem}
\label{Risk-sen NOC FBJ}(Risk-Sensitive necessary optimality conditions): We
assume that $\left(  \mathbf{H}_{4}\right)  $ holds, if $\left(  x^{u}\left(
.\right)  ,y^{u}\left(  .\right)  ,z^{u}\left(  .\right)  ,r^{u}\left(
.,.\right)  ,u\left(  .\right)  \right)  $ is an optimal solution of the
risk-sensitive control problem $\left\{  \left(  \ref{EQ}\right)  ,\left(
\ref{J}\right)  ,\left(  \ref{inf}\right)  \right\}  $, then there exist
$\mathcal{F}_{t}$-adapted processes $\left(  V^{\theta}\left(  t\right)
,l\left(  t\right)  ,L\left(  t,\lambda\right)  \right)  ,$ and $\left(
\widetilde{p}_{2}\left(  t\right)  ,\widetilde{q}_{2}\left(  t\right)
\right)  ,\left(  \widetilde{p}_{3}\left(  t\right)  \right)  ,$ $\left(
\widetilde{\pi}_{2}\left(  t,.\right)  \right)  $ that satisfy $\left(
\ref{system transformed1}\right)  ,$ $\left(  \ref{condition limit1}\right)  $
such that
\[
\partial H^{\theta}\left(  t\right)  \leq0,
\]
for all $u\in\mathcal{U}$, almost every $0\leq t\leq T$ and $\mathbb{P}%
$-almost surely.
\end{theorem}

\begin{proof}
The Hamiltonian $\widetilde{H}^{\theta}$ associated with $\left(
\ref{problem 1}\right)  ,$ is given by%
\[%
\begin{array}
[c]{l}%
\widetilde{H}^{\theta}\left(  t,\xi^{u}\left(  t\right)  ,x^{u}\left(
t\right)  ,y^{u}\left(  t\right)  ,z^{u}\left(  t\right)  ,r^{u}\left(
t,.\right)  ,\overrightarrow{p}^{u}\left(  t\right)  ,\overrightarrow{q}%
^{u}\left(  t\right)  ,\overrightarrow{\pi}^{u}\left(  t,.\right)  \right) \\
=\left\{  \theta V^{\theta}\left(  t\right)  \right\}  H^{\theta}\left(
t,x^{u}\left(  t\right)  ,y^{u}\left(  t\right)  ,z^{u}\left(  t\right)
,r_{t}^{u}\left(  t,.\right)  ,\widetilde{p}_{2}\left(  t\right)
,\widetilde{q}_{2}\left(  t\right)  ,\widetilde{p}_{3}\left(  t\right)
\right. \\
\left.  ,\widetilde{\pi}_{2}\left(  t,.\right)  ,V^{\theta}\left(  t\right)
,l_{2}\left(  t\right)  ,l_{3}\left(  t\right)  ,L_{2}\left(  t,.\right)
,L_{3}\left(  t,.\right)  \right)  ,
\end{array}
\]
and $H^{\theta}$ is the risk-sensitive Hamiltonian given by $\left(
\ref{H risk-sensitive1}\right)  .$ To arrive at a risk-sensitive stochastic
maximum principle expressed in terms of the adjoint processes $\left(
\widetilde{p}_{2},\widetilde{q}_{2}\right)  ,$ $\left(  \widetilde{p}%
_{3},\widetilde{q}_{3}\right)  ,$ $\left(  \widetilde{\pi}_{2},\widetilde{\pi
}_{3}\right)  $and $\left(  V^{\theta},l,L\right)  $, which solve $\left(
\ref{system transformed1}\right)  $. Hence, since $V^{\theta}>0,$ the
variational inequality $\left(  \ref{SMP}\right)  $ translates into $\partial
H^{\theta}\left(  t\right)  \leq0,$ for all $u\in\mathcal{U}$, almost every
$0\leq t\leq T$ and $\mathbb{P}$-almost surely.
\end{proof}

\section{Risk sensitive sufficient optimality conditions}

This section is concerned with a study of the necessary condition of
optimality $\left(  \ref{SMP}\right)  $ when it becomes sufficient.

\begin{theorem}
\label{SOC FBJ}(Risk sensitive sufficient optimality conditions)Assume that
$\Phi(.)$ and $\Psi(.)$ are convex and for all $\left(  x,y,z,r,v\right)  \in%
\mathbb{R}
\times%
\mathbb{R}
\times%
\mathbb{R}
\times\Gamma\times U$ the function $H(t,x,y,z,,r,v,p,q,\pi)$ is convex, and
for any $v\in U$ such that $\mathbb{E}\left\vert v\right\vert
{{}^2}%
<\infty.$ Then, $u$ is an optimal control of the problem $\left\{  \left(
\ref{EQ}\right)  ,\left(  \ref{J}\right)  ,\left(  \ref{inf}\right)  \right\}
$, if it satisfies $\left(  \ref{SMP}\right)  $.
\end{theorem}

\begin{proof}
Let $u$ be an admissible control (candidate to be optimal) for any $v\in U$,
we have%
\[%
\begin{array}
[c]{cc}%
J^{\theta}\left(  v\right)  -J^{\theta}\left(  u\right)  = & \mathbb{E}\left[
\exp\left\{  \theta\Psi\left(  y^{v}\left(  0\right)  \right)  +\theta
\Phi\left(  x^{v}\left(  T\right)  \right)  +\theta\xi^{v}\left(  T\right)
\right\}  \right] \\
& -\mathbb{E}\left[  \exp\left\{  \theta\Psi\left(  y^{u}\left(  0\right)
\right)  +\theta\Phi\left(  x^{u}\left(  T\right)  \right)  +\theta\xi
^{u}\left(  T\right)  \right\}  \right]  .
\end{array}
\]

Since $\Psi$ and $\Phi$ are convex, and applying Taylor's expansion, we get%
\[%
\begin{array}
[c]{l}%
J^{\theta}\left(  v\right)  -J^{\theta}\left(  u\right)  \geq\mathbb{E}\left[
\theta A_{T}\left(  \xi^{v}\left(  T\right)  -\xi^{u}\left(  T\right)
\right)  \right]  +\mathbb{E}\left[  \theta\Phi_{x}\left(  x^{u}\left(
T\right)  \right)  A_{T}\left(  x^{v}\left(  T\right)  -x^{u}\left(  T\right)
\right)  \right] \\
+\mathbb{E}\left[  \theta\Psi_{y}\left(  y^{u}\left(  0\right)  \right)
A_{T}\left(  y^{v}\left(  0\right)  -y^{u}\left(  0\right)  \right)  \right]
.
\end{array}
\]

According to $(\ref{adj1})$, we remark that $p_{1}\left(  T\right)  =\theta
A_{T}$, $p_{2}\left(  T\right)  =\theta\Phi_{x}\left(  x^{u}\left(  T\right)
\right)  A_{T}$ and $\ p_{3}\left(  0\right)  =\theta\Psi_{y}\left(
y^{u}\left(  0\right)  \right)  A_{T},$ then%
\begin{equation}%
\begin{array}
[c]{ll}%
J^{\theta}\left(  v\right)  -J^{\theta}\left(  u\right)  \geq & \mathbb{E}%
\left[  p_{1}\left(  T\right)  \left(  \xi_{T}^{v}-\xi_{T}^{u}\right)
\right]  +\mathbb{E}\left[  p_{2}\left(  T\right)  \left(  x^{v}\left(
T\right)  -x^{u}\left(  T\right)  \right)  \right] \\
& +\mathbb{E}\left[  p_{3}\left(  0\right)  \left(  y^{v}\left(  0\right)
-y^{u}\left(  0\right)  \right)  \right]  .
\end{array}
\label{cost risk suff}%
\end{equation}

We apply It\^{o}'s formula to $p_{1}\left(  t\right)  \left(  \xi^{v}\left(
t\right)  -\xi^{u}\left(  t\right)  \right)  $,%
\[%
\begin{array}
[c]{ll}%
d\left(  p_{1}\left(  t\right)  \left(  \xi^{v}\left(  t\right)  -\xi
^{u}\left(  t\right)  \right)  \right)  = & \left(  \xi^{v}\left(  t\right)
-\xi^{u}\left(  t\right)  \right)  dp_{1}\left(  t\right)  +p_{1}\left(
t\right)  d\left(  \xi^{v}\left(  t\right)  -\xi^{u}\left(  t\right)  \right)
\\
& +\left\langle \left(  \xi^{v}-\xi^{u}\right)  ,p_{1}\right\rangle _{t}dt+%
{\displaystyle\int\nolimits_{\Gamma}}
\left\langle \left(  \xi^{v}-\xi^{u}\right)  ,p_{1}\right\rangle _{t}m\left(
d\lambda\right)  dt
\end{array}
,
\]

then%
\[%
\begin{array}
[c]{ll}%
{\displaystyle\int\nolimits_{0}^{T}}
\left(  p_{1}\left(  t\right)  \left(  \xi^{v}\left(  t\right)  -\xi
^{u}\left(  t\right)  \right)  \right)  dt & =%
{\displaystyle\int\nolimits_{0}^{T}}
\left(  \xi^{v}\left(  t\right)  -\xi^{u}\left(  t\right)  \right)
dp_{1}\left(  t\right)  +%
{\displaystyle\int\nolimits_{0}^{T}}
p_{1}\left(  t\right)  d\left(  \xi^{v}\left(  t\right)  -\xi^{u}\left(
t\right)  \right) \\
& +%
{\displaystyle\int\nolimits_{0}^{T}}
\left\langle \left(  \xi^{v}-\xi^{u}\right)  ,p_{1}\right\rangle _{t}dt+%
{\displaystyle\int\nolimits_{0}^{T}}
{\displaystyle\int\nolimits_{\Gamma}}
\left\langle \left(  \xi^{v}-\xi^{u}\right)  ,p_{1}\right\rangle _{t}m\left(
d\lambda\right)  dt\\
& =%
{\displaystyle\int\nolimits_{0}^{T}}
\left(  f\left(  t,x^{v}\left(  t\right)  ,y^{v}\left(  t\right)
,z^{v}\left(  t\right)  ,r^{v}\left(  t,.\right)  ,v_{t}\right)  -\right. \\
& \left.  f\left(  t,x^{u}\left(  t\right)  ,y^{u}\left(  t\right)
,z^{u}\left(  t\right)  ,r^{u}\left(  t,.\right)  ,u_{t}\right)  \right)
q_{1}\left(  t\right)  dW_{t}\\
& +%
{\displaystyle\int\nolimits_{0}^{T}}
{\displaystyle\int_{\Gamma}}
\left(  f\left(  t,x^{v}\left(  t\right)  ,y^{v}\left(  t\right)
,z^{v}\left(  t\right)  ,r^{v}\left(  t,\lambda\right)  ,v_{t}\right)
-\right. \\
& \left.  f\left(  t,x^{u}\left(  t\right)  ,y^{u}\left(  t\right)
,z^{u}\left(  t\right)  ,r^{u}\left(  t,\lambda\right)  ,u_{t}\right)
\right)  \pi_{1}\left(  \lambda,t\right)  \widetilde{N}\left(  d\lambda
,dt\right) \\
& +%
{\displaystyle\int\nolimits_{0}^{T}}
\left(  f\left(  t,x^{v}\left(  t\right)  ,y^{v}\left(  t\right)
,z^{v}\left(  t\right)  ,r^{v}\left(  t,.\right)  ,v_{t}\right)  -\right. \\
& \left.  f\left(  t,x^{u}\left(  t\right)  ,y^{u}\left(  t\right)
,z^{u}\left(  t\right)  ,r^{u}\left(  t,.\right)  ,u_{t}\right)  \right)
p_{1}\left(  t\right)  dt
\end{array}
,
\]

We apply expectation, we get%
\begin{equation}%
\begin{array}
[c]{ll}%
\mathbb{E}\left[  p_{1}\left(  T\right)  \left(  \xi^{v}\left(  T\right)
-\xi^{u}\left(  T\right)  \right)  \right]  & =\mathbb{E}\left[
{\displaystyle\int\nolimits_{0}^{T}}
\left(  f\left(  t,x^{v}\left(  t\right)  ,y^{v}\left(  t\right)
,z^{v}\left(  t\right)  ,r^{v}\left(  t,.\right)  ,v_{t}\right)  -\right.
\right. \\
& \left.  \left.  f\left(  t,x^{u}\left(  t\right)  ,y^{u}\left(  t\right)
,z^{u}\left(  t\right)  ,r^{u}\left(  t,.\right)  ,u_{t}\right)  \right)
p_{1}\left(  t\right)  dt\right]
\end{array}
. \label{p1adj}%
\end{equation}

And we apply also It\^{o}'s formula to $p_{2}\left(  t\right)  \left(
x^{v}\left(  t\right)  -x^{u}\left(  t\right)  \right)  $
\[%
\begin{array}
[c]{ll}%
d\left(  p_{2}\left(  t\right)  \left(  x^{v}\left(  t\right)  -x^{u}\left(
t\right)  \right)  \right)  = & \left(  x^{v}\left(  t\right)  -x^{u}\left(
t\right)  \right)  dp_{2}\left(  t\right)  +p_{2}\left(  t\right)  d\left(
x^{v}\left(  t\right)  -x^{u}\left(  t\right)  \right) \\
& +\left\langle x^{v}-x^{u},p_{2}\right\rangle _{t}dt+%
{\displaystyle\int\nolimits_{\Gamma}}
\left\langle x^{v}-x^{u},p_{2}\right\rangle _{t}m\left(  d\lambda\right)  dt
\end{array}
,
\]

then%
\[%
\begin{array}
[c]{ll}%
{\displaystyle\int\nolimits_{0}^{T}}
d\left(  p_{2}\left(  t\right)  \left(  x^{v}\left(  t\right)  -x^{u}\left(
t\right)  \right)  \right)  = &
{\displaystyle\int\nolimits_{0}^{T}}
\left(  b\left(  t,x^{v}\left(  t\right)  ,y^{v}\left(  t\right)
,z^{v}\left(  t\right)  ,r^{v}\left(  t,.\right)  ,v_{t}\right)  -\right. \\
& \left.  b\left(  t,x^{u}\left(  t\right)  ,y^{u}\left(  t\right)
,z^{u}\left(  t\right)  ,r^{u}\left(  t,.\right)  ,u_{t}\right)  \right)
p_{2}\left(  t\right)  dt\\
& +%
{\displaystyle\int\nolimits_{0}^{T}}
\left(  \sigma\left(  t,x^{v}\left(  t\right)  ,y^{v}\left(  t\right)
,z^{v}\left(  t\right)  ,r^{v}\left(  t,.\right)  ,v_{t}\right)  -\right. \\
& \left.  \sigma\left(  t,x^{u}\left(  t\right)  ,y^{u}\left(  t\right)
,z^{u}\left(  t\right)  ,r^{u}\left(  t,.\right)  ,u_{t}\right)  \right)
p_{2}\left(  t\right)  dW_{t}\\
& +%
{\displaystyle\int\nolimits_{0}^{T}}
{\displaystyle\int\nolimits_{\Gamma}}
\left(  \gamma\left(  t,x^{v}\left(  t\right)  ,y^{v}\left(  t\right)
,z^{v}\left(  t\right)  ,r^{v}\left(  t,\lambda\right)  ,v_{t}\right)
-\right. \\
& \left.  \gamma\left(  t,x^{u}\left(  t\right)  ,y^{u}\left(  t\right)
,z^{u}\left(  t\right)  ,r^{u}\left(  t,\lambda\right)  ,u_{t}\right)
\right)  p_{2}\left(  t\right)  \widetilde{N}\left(  d\lambda,dt\right) \\
& +%
{\displaystyle\int\nolimits_{0}^{T}}
-\left(  f_{x}\left(  t\right)  p_{1}+b_{x}\left(  t\right)  p_{2}+\sigma
_{x}\left(  t\right)  q_{2}\right. \\
& \left.  +g_{x}\left(  t\right)  p_{3}+%
{\displaystyle\int_{\Gamma}}
\gamma_{x}\left(  t-,\lambda\right)  \pi_{2}\left(  \lambda,t\right)  m\left(
d\lambda\right)  \right)  \left(  x_{t}^{v}-x_{t}^{u}\right)  dt\\
& +%
{\displaystyle\int\nolimits_{0}^{T}}
q_{2}\left(  t\right)  \left(  x_{t}^{v}-x_{t}^{u}\right)  dB_{t}+%
{\displaystyle\int\nolimits_{0}^{T}}
{\displaystyle\int\nolimits_{\Gamma}}
\pi_{2}\left(  \lambda,t\right)  \left(  x_{t}^{v}-x_{t}^{u}\right)
\widetilde{N}\left(  d\lambda,dt\right) \\
& +%
{\displaystyle\int\nolimits_{0}^{T}}
{\displaystyle\int\nolimits_{\Gamma}}
\left(  \gamma\left(  t,x^{v}\left(  t\right)  ,y^{v}\left(  t\right)
,z^{v}\left(  t\right)  ,r^{v}\left(  t,\lambda\right)  ,v_{t}\right)
-\right. \\
& \left.  \gamma\left(  t,x^{u}\left(  t\right)  ,y^{u}\left(  t\right)
,z^{u}\left(  t\right)  ,r^{u}\left(  t,\lambda\right)  ,u_{t}\right)
\right)  \pi_{2}\left(  \lambda,t\right)  m\left(  d\lambda\right)  dt
\end{array}
\]

We apply expectation, we get%
\begin{equation}%
\begin{array}
[c]{l}%
\mathbb{E}\left[  p_{2}\left(  T\right)  \left(  x^{v}\left(  T\right)
-x^{u}\left(  T\right)  \right)  \right]  =\\
\mathbb{E}\left[
{\displaystyle\int\nolimits_{0}^{T}}
-\left(  f_{x}\left(  t\right)  p_{1}+b_{x}\left(  t\right)  p_{2}+\sigma
_{x}\left(  t\right)  q_{2}\right.  \right. \\
\left.  \left.  +g_{x}\left(  t\right)  p_{3}+%
{\displaystyle\int_{\Gamma}}
\gamma_{x}\left(  t-,\lambda\right)  \pi_{2}\left(  \lambda,t\right)  m\left(
d\lambda\right)  \right)  \left(  x_{t}^{v}-x_{t}^{u}\right)  dt\right] \\
+\mathbb{E}\left[
{\displaystyle\int\nolimits_{0}^{T}}
\left(  b\left(  t,x^{v}\left(  t\right)  ,y^{v}\left(  t\right)
,z^{v}\left(  t\right)  ,r^{v}\left(  t,.\right)  ,v_{t}\right)  -\right.
\right. \\
\left.  \left.  b\left(  t,x^{u}\left(  t\right)  ,y^{u}\left(  t\right)
,z^{u}\left(  t\right)  ,r^{u}\left(  t,.\right)  ,u_{t}\right)  \right)
p_{2}\left(  t\right)  dt\right] \\
+\mathbb{E}\left[
{\displaystyle\int\nolimits_{0}^{T}}
\left(  \sigma\left(  t,x^{v}\left(  t\right)  ,y^{v}\left(  t\right)
,z^{v}\left(  t\right)  ,r^{v}\left(  t,.\right)  ,v_{t}\right)  -\right.
\right. \\
\left.  \left.  \sigma\left(  t,x^{u}\left(  t\right)  ,y^{u}\left(  t\right)
,z^{u}\left(  t\right)  ,r^{u}\left(  t,.\right)  ,u_{t}\right)  \right)
q_{2}\left(  t\right)  dt\right] \\
+\mathbb{E}\left[
{\displaystyle\int\nolimits_{0}^{T}}
{\displaystyle\int_{\Gamma}}
\left(  \gamma\left(  t,x^{v}\left(  t-\right)  ,y^{v}\left(  t-\right)
,z^{v}\left(  t-\right)  ,r^{v}\left(  t-,\lambda\right)  ,v_{t-}\right)
-\right.  \right. \\
\left.  \left.  \gamma\left(  t,x^{u}\left(  t-\right)  ,y^{u}\left(
t-\right)  ,z^{u}\left(  t-\right)  ,r^{u}\left(  t-,\lambda\right)
,u_{t-}\right)  \right)  \pi\left(  t,\lambda\right)  m\left(  d\lambda
\right)  dt\right]  ,
\end{array}
\label{p2adj}%
\end{equation}

We apply also It\^{o}'s formula to $p_{3}\left(  t\right)  \left(
y^{v}\left(  t\right)  -y^{u}\left(  t\right)  \right)  $%
\[%
\begin{array}
[c]{ll}%
d\left(  p_{3}\left(  t\right)  \left(  y^{v}\left(  t\right)  -y^{u}\left(
t\right)  \right)  \right)  = & \left(  y^{v}\left(  t\right)  -y^{u}\left(
t\right)  \right)  dp_{3}\left(  t\right)  +p_{3}\left(  t\right)  d\left(
y^{v}\left(  t\right)  -y^{u}\left(  t\right)  \right) \\
& +\left\langle y^{v}-y^{u},p_{3}\right\rangle _{t}dt+%
{\displaystyle\int\nolimits_{\Gamma}}
\left\langle y^{v}-y^{u},p_{3}\right\rangle _{t}m\left(  d\lambda\right)  dt
\end{array}
,
\]

then%
\[%
\begin{array}
[c]{ll}%
{\displaystyle\int\nolimits_{0}^{T}}
d\left(  p_{3}\left(  t\right)  \left(  y^{v}\left(  t\right)  -y^{u}\left(
t\right)  \right)  \right)  = &
{\displaystyle\int\nolimits_{0}^{T}}
\left(  g\left(  t,x^{v}\left(  t\right)  ,y^{v}\left(  t\right)
,z^{v}\left(  t\right)  ,r^{v}\left(  t,.\right)  ,v_{t}\right)  -\right. \\
& \left.  g\left(  t,x^{u}\left(  t\right)  ,y^{u}\left(  t\right)
,z^{u}\left(  t\right)  ,r^{u}\left(  t,.\right)  ,u_{t}\right)  \right)
p_{3}\left(  t\right)  dt\\
&
{\displaystyle\int\nolimits_{0}^{T}}
\left(  z^{v}\left(  t\right)  -z^{u}\left(  t\right)  \right)  p_{3}\left(
t\right)  dW_{t}\\
& +%
{\displaystyle\int\nolimits_{0}^{T}}
{\displaystyle\int_{\Gamma}}
\left(  r^{v}\left(  t,\lambda\right)  -r^{u}\left(  t,\lambda\right)
\right)  p_{3}\left(  t\right)  \widetilde{N}\left(  dt,d\lambda\right) \\
& +%
{\displaystyle\int\nolimits_{0}^{T}}
-\left(  f_{y}\left(  t\right)  p_{1}+b_{y}\left(  t\right)  p_{2}+\sigma
_{y}\left(  t\right)  q_{2}\right. \\
& \left.  +g_{y}\left(  t\right)  p_{3}+%
{\displaystyle\int_{\Gamma}}
\gamma_{y}\left(  t-,\lambda\right)  \pi_{2}\left(  \lambda,t\right)  m\left(
d\lambda\right)  \right)  \left(  y_{t}^{v}-y_{t}^{u}\right)  dt\\
& +%
{\displaystyle\int\nolimits_{0}^{T}}
-\left(  f_{z}\left(  t\right)  p_{1}+b_{z}\left(  t\right)  p_{2}+\sigma
_{z}\left(  t\right)  q_{2}+g_{z}\left(  t\right)  p_{3}\right. \\
& \left.  +%
{\displaystyle\int_{\Gamma}}
\gamma_{z}\left(  t-,\lambda\right)  \pi_{2}\left(  \lambda,t\right)  m\left(
d\lambda\right)  \right)  \left(  y_{t}^{v}-y_{t}^{u}\right)  dW_{t}\\
& +%
{\displaystyle\int\nolimits_{0}^{T}}
{\displaystyle\int_{\Gamma}}
-\left(  f_{r}\left(  t\right)  p_{1}+b_{r}\left(  t\right)  p_{2}+\sigma
_{r}\left(  t\right)  q_{2}+g_{r}\left(  t\right)  p_{3}\right. \\
& \left.  +%
{\displaystyle\int_{\Gamma}}
\gamma_{r}\left(  t-,\lambda\right)  \pi_{2}\left(  \lambda,t\right)  m\left(
d\lambda\right)  \right)  \left(  y_{t}^{v}-y_{t}^{u}\right)  \widetilde
{N}\left(  d\lambda,dt\right)
\end{array}
\]

We apply expectation, We get%
\begin{equation}%
\begin{array}
[c]{l}%
\mathbb{E}\left[  p_{3}\left(  0\right)  \left(  y^{v}\left(  0\right)
-y^{u}\left(  0\right)  \right)  \right]  =\\
\mathbb{E}\left[
{\displaystyle\int\nolimits_{0}^{T}}
\left(  g\left(  t,x^{v}\left(  t\right)  ,y^{v}\left(  t\right)
,z^{v}\left(  t\right)  ,r^{v}\left(  t,.\right)  ,v_{t}\right)  -\right.
\right. \\
\left.  \left.  g\left(  t,x^{u}\left(  t\right)  ,y^{u}\left(  t\right)
,z^{u}\left(  t\right)  ,r^{u}\left(  t,.\right)  ,u_{t}\right)  \right)
p_{3}\left(  t\right)  dt\right] \\
-\mathbb{E}\left[
{\displaystyle\int\nolimits_{0}^{T}}
\left(  f_{y}\left(  t\right)  p_{1}\left(  t\right)  +b_{y}\left(  t\right)
p_{2}\left(  t\right)  +g_{y}\left(  t\right)  p_{3}\left(  t\right)
+\sigma_{y}\left(  t\right)  q_{2}\left(  t\right)  \right.  \right. \\
\left.  \left.  +%
{\displaystyle\int_{\Gamma}}
\gamma_{y}\left(  t-,\lambda\right)  \pi_{2}\left(  t,\lambda\right)  m\left(
d\lambda\right)  \right)  \left(  y^{v}\left(  t\right)  -y^{u}\left(
t\right)  \right)  dt\right] \\
-\mathbb{E}\left[
{\displaystyle\int\nolimits_{0}^{T}}
\left(  f_{z}\left(  t\right)  p_{1}\left(  t\right)  +b_{z}\left(  t\right)
p_{2}\left(  t\right)  +g_{z}\left(  t\right)  p_{3}\left(  t\right)
+\sigma_{z}\left(  t\right)  q_{2}\left(  t\right)  \right.  \right. \\
\left.  \left.  +%
{\displaystyle\int_{\Gamma}}
\gamma_{z}\left(  t-,\lambda\right)  \pi_{2}\left(  t,\lambda\right)  m\left(
d\lambda\right)  \right)  \left(  z^{v}\left(  t\right)  -z^{u}\left(
t\right)  \right)  dt\right] \\
-\mathbb{E}\left[
{\displaystyle\int\nolimits_{0}^{T}}
{\displaystyle\int_{\Gamma}}
\left(  f_{r}\left(  t\right)  p_{1}\left(  t\right)  +b_{r}\left(  t\right)
p_{2}\left(  t\right)  +g_{r}\left(  t\right)  p_{3}\left(  t\right)
+\sigma_{r}\left(  t\right)  q_{2}\left(  t\right)  \right.  \right. \\
\left.  \left.  +\gamma_{r}\left(  t-,\lambda\right)  \pi_{2}\left(
t,\lambda\right)  \right)  \left(  r_{t}^{v}\left(  \lambda\right)  -r_{t}%
^{u}\left(  \lambda\right)  \right)  m\left(  d\lambda\right)  dt\right]  .
\end{array}
\label{p3adj}%
\end{equation}

By replacing $\left(  \ref{p1adj}\right)  $ ,$\left(  \ref{p2adj}\right)  $
and $\left(  \ref{p3adj}\right)  $ into $\left(  \ref{cost risk suff}\right)
,$ we have%
\[%
\begin{array}
[c]{l}%
J^{\theta}\left(  v\right)  -J^{\theta}\left(  u\right) \\
\geq\mathbb{E}\left[  \int_{0}^{T}\left(  \widetilde{H}^{\theta}\left(
t,x^{v}\left(  t\right)  ,y^{v}\left(  t\right)  ,z^{v}\left(  t\right)
,r^{v}\left(  t,.\right)  ,v_{t},p\left(  t\right)  ,q\left(  t\right)
,\pi\left(  t,.\right)  \right)  -\right.  \right. \\
\left.  \left.  \widetilde{H}^{\theta}\left(  t,x^{u}\left(  t\right)
,y^{u}\left(  t\right)  ,z^{u}\left(  t\right)  ,r^{u}\left(  t,.\right)
,u_{t},p^{u}\left(  t\right)  ,q^{u}\left(  t\right)  ,\pi\left(  t,.\right)
\right)  \right)  dt\right] \\
-\mathbb{E}\left[
{\displaystyle\int\nolimits_{0}^{T}}
\widetilde{H}_{x}^{\theta}\left(  t,x^{u}\left(  t\right)  ,y^{u}\left(
t\right)  ,z^{u}\left(  t\right)  ,r^{u}\left(  t,.\right)  ,u_{t}%
,p^{u}\left(  t\right)  ,q^{u}\left(  t\right)  ,\pi\left(  t,.\right)
\right)  \left(  x^{v}\left(  t\right)  -x^{u}\left(  t\right)  \right)
dt\right] \\
-\mathbb{E}\left[
{\displaystyle\int\nolimits_{0}^{T}}
\widetilde{H}_{y}^{\theta}\left(  t,x^{u}\left(  t\right)  ,y^{u}\left(
t\right)  ,z^{u}\left(  t\right)  ,r^{u}\left(  t,.\right)  ,u_{t}%
,p^{u}\left(  t\right)  ,q^{u}\left(  t\right)  ,\pi\left(  t,.\right)
\right)  \left(  y^{v}\left(  t\right)  -y^{u}\left(  t\right)  \right)
dt\right] \\
-E\left[
{\displaystyle\int\nolimits_{0}^{T}}
\widetilde{H}_{z}^{\theta}\left(  t,x^{u}\left(  t\right)  ,y^{u}\left(
t\right)  ,z^{u}\left(  t\right)  ,r^{u}\left(  t,.\right)  ,u_{t}%
,p^{u}\left(  t\right)  ,q^{u}\left(  t\right)  ,\pi\left(  t,.\right)
\right)  \left(  z^{v}\left(  t\right)  -z^{u}\left(  t\right)  \right)
dt\right] \\
-E\left[
{\displaystyle\int\nolimits_{0}^{T}}
\nabla\widetilde{H}_{r}^{\theta}\left(  t,x^{u}\left(  t\right)  ,y^{u}\left(
t\right)  ,z^{u}\left(  t\right)  ,r^{u}\left(  t,.\right)  ,u_{t}%
,p^{u}\left(  t\right)  ,q^{u}\left(  t\right)  ,\pi\left(  t,.\right)
\right)  \right. \\
\left.  \left(  r^{v}\left(  t,.\right)  -r^{u}\left(  t,.\right)  \right)
dt\right]  .
\end{array}
\]

Since the Hamiltonian $H$ is concave with respect to $\left(
x,y,z,r,v\right)  $, we have%
\[%
\begin{array}
[c]{l}%
\mathbb{E}\left[
{\displaystyle\int\nolimits_{0}^{T}}
\widetilde{H}_{v}^{\theta}\left(  t,x^{u}\left(  t\right)  ,y^{u}\left(
t\right)  ,z^{u}\left(  t\right)  ,r^{u}\left(  t,.\right)  ,u_{t}%
,p^{u}\left(  t\right)  ,q^{u}\left(  t\right)  ,\pi\left(  t,.\right)
\right)  \left(  v_{t}-u_{t}\right)  dt\right] \\
\leq\mathbb{E}\left[  \int_{0}^{T}\left(  \widetilde{H}^{\theta}\left(
t,x^{v}\left(  t\right)  ,y^{v}\left(  t\right)  ,z^{v}\left(  t\right)
,r^{v}\left(  t,.\right)  ,v_{t},p^{u}\left(  t\right)  ,q^{u}\left(
t\right)  ,\pi^{u}\left(  t,.\right)  \right)  -\right.  \right. \\
\left.  \left.  \widetilde{H}^{\theta}\left(  t,x^{v}\left(  t\right)
,y^{v}\left(  t\right)  ,z^{v}\left(  t\right)  ,r^{v}\left(  t,.\right)
,u_{t},p^{u}\left(  t\right)  ,q^{u}\left(  t\right)  ,\pi^{u}\left(
t,.\right)  \right)  \right)  dt\right] \\
+\mathbb{E}\left[
{\displaystyle\int\nolimits_{0}^{T}}
\widetilde{H}_{x}^{\theta}\left(  t,x^{v}\left(  t\right)  ,y^{v}\left(
t\right)  ,z^{v}\left(  t\right)  ,r^{v}\left(  t,.\right)  ,v_{t}%
,p^{u}\left(  t\right)  ,q^{u}\left(  t\right)  ,\pi^{u}\left(  t,.\right)
\right)  \left(  x^{v}\left(  t\right)  -x^{u}\left(  t\right)  \right)
dt\right] \\
+\mathbb{E}\left[
{\displaystyle\int\nolimits_{0}^{T}}
\widetilde{H}_{y}^{\theta}\left(  t,x^{v}\left(  t\right)  ,y^{v}\left(
t\right)  ,z^{v}\left(  t\right)  ,r^{v}\left(  t,.\right)  ,v_{t}%
,p^{u}\left(  t\right)  ,q^{u}\left(  t\right)  ,\pi^{u}\left(  t,.\right)
\right)  \left(  y^{v}\left(  t\right)  -y^{u}\left(  t\right)  \right)
dt\right] \\
+E\left[
{\displaystyle\int\nolimits_{0}^{T}}
\widetilde{H}_{z}^{\theta}\left(  t,x^{v}\left(  t\right)  ,y^{v}\left(
t\right)  ,z^{v}\left(  t\right)  ,r^{v}\left(  t,.\right)  ,v_{t}%
,p^{u}\left(  t\right)  ,q^{u}\left(  t\right)  ,\pi^{u}\left(  t,.\right)
\right)  \left(  z^{v}\left(  t\right)  -z^{u}\left(  t\right)  \right)
dt\right] \\
+E\left[
{\displaystyle\int\nolimits_{0}^{T}}
\nabla\widetilde{H}_{r}^{\theta}\left(  t,x^{v}\left(  t\right)  ,y^{v}\left(
t\right)  ,z^{v}\left(  t\right)  ,r^{v}\left(  t,.\right)  ,v_{t}%
,p^{u}\left(  t\right)  ,q^{u}\left(  t\right)  ,\pi^{u}\left(  t,.\right)
\right)  \right. \\
\left.  \left(  r^{v}\left(  t,.\right)  -r^{u}\left(  t,.\right)  \right)
dt\right]  .
\end{array}
\]

Then
\[%
\begin{array}
[c]{l}%
J^{\theta}\left(  v\right)  -J^{\theta}\left(  u\right) \\
\geq\mathbb{E}\left[
{\displaystyle\int\nolimits_{0}^{T}}
\widetilde{H}_{v}^{\theta}\left(  t,x^{v}\left(  t\right)  ,y^{v}\left(
t\right)  ,z^{v}\left(  t\right)  ,r^{v}\left(  t,.\right)  ,u_{t}%
,p^{u}\left(  t\right)  ,q^{u}\left(  t\right)  ,\pi^{u}\left(  t,.\right)
\right)  \left(  v_{t}-u_{t}\right)  dt\right]  .
\end{array}
\]

In virtue of the necessary condition of optimality $\left(  \ref{SMP}\right)
$ the last inequality implies that $J^{\theta}\left(  v\right)  -J^{\theta
}\left(  u\right)  \geq0.$ Then, the theorem is improved.
\end{proof}

\section{Example: Mean-Variance (Cash-flow):}

Now we return to the problem of optimal portfolio stated in the motivating
example, and apply the risk sensitive necessary optimality condition (Theorem
$\ref{Risk-sen NOC FBJ}$).

Our state dynamics is%
\begin{equation}
\left\{
\begin{array}
[c]{l}%
dx\left(  t\right)  =\left(  \rho v\left(  t\right)  -cx\left(  t\right)
\right)  dt+\sigma v\left(  t\right)  dW\left(  t\right)  +%
{\displaystyle\int_{\Gamma}}
v\left(  t\right)  \left(  1+r\left(  t,\lambda\right)  \right)  \widetilde
{N}\left(  d\lambda,dt\right)  ,\\
x\left(  0\right)  =m_{0}=d,
\end{array}
\right.  \label{forwardexemple}%
\end{equation}

and
\begin{equation}
\left\{
\begin{array}
[c]{l}%
dy\left(  t\right)  =\left(  \rho v\left(  t\right)  -cx\left(  t\right)
+\lambda y\left(  t\right)  \right)  dt+z\left(  t\right)  dW\left(  t\right)
+%
{\displaystyle\int_{\Gamma}}
r\left(  t,\lambda\right)  \widetilde{N}\left(  d\lambda,dt\right)  ,\\
y\left(  T\right)  =0=a.
\end{array}
\right.  \label{backwardexemple}%
\end{equation}

The cost functional is%
\[
J^{\theta}\left(  v\left(  .\right)  \right)  =\exp\left\{  \theta
\widetilde{J}^{\theta}\left(  v\left(  .\right)  \right)  \right\}  ,
\]
where $\widetilde{J}$ is the neutral cost functional given by the following
expected with an exponential form see section 1.2.3
\begin{equation}
\widetilde{J}^{\theta}\left(  v\left(  .\right)  \right)  =\frac{\theta}%
{2}\mathbb{E}\left(  \Psi_{T}-a\right)
{{}^2}%
+\mathbb{E}\left(  \Psi_{T}\right)  +o\left(  \theta%
{{}^2}%
\right)  , \label{costexemple}%
\end{equation}

Where $\Psi_{T}=\left(  x_{T}+y_{0}\right)  $. The investor wants to minimize
$\left(  \ref{costexemple}\right)  $ subject to $\left(  \ref{forwardexemple}%
\right)  $ and $\left(  \ref{backwardexemple}\right)  $ by taking $v\left(
.\right)  $ over $\mathcal{U}$, the mean--variance portfolio selection problem
is to find $u(t)$ which minimize
\[
\mathbb{V}ar(\Psi_{T})=\mathbb{E}\left(  x_{T}+y_{0}-a\right)
{{}^2}%
\]

The Hamiltonian function $\left(  \ref{H risk-sensitive1}\right)  $ gets the
form%
\[%
\begin{array}
[c]{ll}%
H^{\theta}\left(  t\right)  & :=H^{\theta}\left(  t,x\left(  t\right)
,y\left(  t\right)  ,z\left(  t\right)  ,r\left(  t,\lambda\right)
,\widetilde{p}_{2}\left(  t\right)  ,\widetilde{q}_{2}\left(  t\right)
,\widetilde{p}_{3}\left(  t\right)  ,\widetilde{\pi}_{2}\left(  t,.\right)
,l\left(  t\right)  ,L\left(  t,.\right)  ,v_{t}\right) \\
& =f\left(  t\right)  +b\left(  t\right)  \widetilde{p}_{2}\left(  t\right)
+\sigma\left(  t\right)  \widetilde{q}_{2}\left(  t\right)  +\left\{  g\left(
t\right)  -\theta l\left(  t\right)  z\left(  t\right)  \right\}
\widetilde{p}_{3}\left(  t\right) \\
& +%
{\displaystyle\int_{\Gamma}}
\left\{  \gamma\left(  t-,\lambda\right)  \widetilde{\pi}_{2}\left(
t,\lambda\right)  -\left(  g\left(  t\right)  -\theta L\left(  t,\lambda
\right)  r\left(  t,\lambda\right)  \right)  \widetilde{p}_{3}\left(
t\right)  \right\}  m\left(  d\lambda\right) \\
& =\left(  \rho v\left(  t\right)  -cx\left(  t\right)  \right)  \widetilde
{p}_{2}\left(  t\right)  +\sigma v\left(  t\right)  \widetilde{q}_{2}\left(
t\right)  +\left\{  \left(  \rho v\left(  t\right)  -cx\left(  t\right)
+\lambda y\left(  t\right)  \right)  -\theta l\left(  t\right)  z\left(
t\right)  \right\}  \widetilde{p}_{3}\left(  t\right) \\
& -%
{\displaystyle\int_{\Gamma}}
\left\{  v\left(  t\right)  \left(  1+r\left(  t,\lambda\right)  \right)
\widetilde{\pi}_{2}\left(  t,\lambda\right)  -\left(  \left(  \rho v\left(
t\right)  -cx\left(  t\right)  +\lambda y\left(  t\right)  \right)  -\theta
L\left(  t,\lambda\right)  r\left(  t,\lambda\right)  \right)  \right. \\
& \left.  \widetilde{p}_{3}\left(  t\right)  \right\}  m\left(  d\lambda
\right)  .
\end{array}
\]

Then, to get the optimal control, the derivative of the above Hamiltonian with
respect to the control process gives us%

\begin{equation}%
\begin{array}
[c]{ll}%
H_{u}^{\theta}\left(  t\right)  & :=H_{u}^{\theta}\left(  t,x\left(  t\right)
,y\left(  t\right)  ,z\left(  t\right)  ,r\left(  t,.\right)  ,\widetilde
{p}_{2}\left(  t\right)  ,\widetilde{q}_{2}\left(  t\right)  ,\widetilde
{p}_{3}\left(  t\right)  ,\widetilde{\pi}_{2}\left(  t,.\right)  ,l\left(
t\right)  ,L\left(  t,.\right)  ,v_{t}\right) \\
& =\rho\widetilde{p}_{2}\left(  t\right)  +\sigma\widetilde{q}_{2}\left(
t\right)  +%
{\displaystyle\int_{\Gamma}}
\left(  1+r\left(  t,\lambda\right)  \right)  \widetilde{\pi}_{2}\left(
t,\lambda\right)  m\left(  d\lambda\right) \\
& =0
\end{array}
\label{the Hamiltonien derive}%
\end{equation}

Let $\left(  x^{u}\left(  t\right)  ,u\left(  t\right)  \right)  $ be an
optimal pair, the adjoint equation $\left(  \ref{P2tilde}\right)  ,$ is given
by%
\[
\left\{
\begin{array}
[c]{ll}%
d\widetilde{p}_{2}^{u}\left(  t\right)  & =c\left(  \widetilde{p}_{2}%
^{u}\left(  t\right)  +c\widetilde{p}_{3}^{u}\left(  t\right)  \right)
dt+\left(  \widetilde{q}_{2}^{u}\left(  t\right)  -\theta l_{2}\left(
t\right)  \widetilde{p}_{2}^{u}\left(  t\right)  \right)  dW^{\theta}\left(
t\right) \\
& +%
{\displaystyle\int_{\Gamma}}
\left(  \widetilde{\pi}_{2}\left(  t,\lambda\right)  -\theta L_{2}\left(
t,\lambda\right)  \widetilde{p}_{2}^{u}\left(  t\right)  \right)
\widetilde{N}^{\theta}\left(  d\lambda,dt\right)  ,\\
\widetilde{p}_{2}^{u}\left(  T\right)  & =1+\theta\left(  x_{T}-y_{0}%
-a\right)  .
\end{array}
\right.
\]

By using of $\ref{Girsanov trans},$ we get%

\begin{equation}
\left\{
\begin{array}
[c]{ll}%
d\widetilde{p}_{2}^{u}\left(  t\right)  & =\left\{  \left(  c+\theta^{2}%
l^{2}\left(  t\right)  +\int_{\Gamma}\theta^{2}L^{2}\left(  t,\lambda\right)
m\left(  d\lambda\right)  \right)  \widetilde{p}_{2}^{u}\left(  t\right)
+c\widetilde{p}_{3}^{u}\left(  t\right)  -\theta l\left(  t\right)
\widetilde{q}_{2}^{u}\right. \\
& \left.  -\int_{\Gamma}\theta L\left(  t,\lambda\right)  \widetilde{\pi}%
_{2}^{u}\left(  t,\lambda\right)  m\left(  d\lambda\right)  \right\}
dt+\left(  \widetilde{q}_{2}^{u}\left(  t\right)  -\theta l_{2}\left(
t\right)  \widetilde{p}_{2}^{u}\left(  t\right)  \right)  dW\left(  t\right)
\\
& +%
{\displaystyle\int_{\Gamma}}
\left(  \widetilde{\pi}_{2}\left(  t,\lambda\right)  -\theta L_{2}\left(
t,\lambda\right)  \widetilde{p}_{2}^{u}\left(  t\right)  \right)
\widetilde{N}\left(  d\lambda,dt\right)  ,\\
\widetilde{p}_{2}^{u}\left(  T\right)  & =1+\theta\left(  x_{T}-y_{0}%
-a\right)  .
\end{array}
\right.  \label{p2exemple}%
\end{equation}

Therefore, an optimal solution $\left(  x_{t}^{u},\widetilde{p}_{2}^{u}\left(
t\right)  ,u_{t}\right)  $ can be obtained by solving the system FBSDE with
jumps diffusion $\left(  \ref{forwardexemple}\right)  $ and $\left(
\ref{p2exemple}\right)  ,$ unfortunately, in such system is difficult to find
the explicit solution, to this end we use the similar technique as in
\cite{Yong Zhoo} see also \cite{Yong}, we conjecture the solution to $\left(
\ref{forwardexemple}\right)  $ and $\left(  \ref{p2exemple}\right)  $ is
related by%
\begin{equation}
\widetilde{p}_{2}^{u}\left(  t\right)  =A\left(  t\right)  x^{u}\left(
t\right)  +B\left(  t\right)  , \label{conjectionexemple}%
\end{equation}

for some deterministic differentiable functions $A\left(  t\right)  $ and
$B\left(  t\right)  .$ Applying It\^{o}'s formula to $\left(
\ref{conjectionexemple}\right)  ,$ we get%
\begin{equation}
\left\{
\begin{array}
[c]{ll}%
d\widetilde{p}_{2}^{u}\left(  t\right)  = & \left[  \overset{\bullet}%
{A}\left(  t\right)  x^{u}\left(  t\right)  +A(t)(\rho u_{t}-cx^{u}\left(
t\right)  )+\overset{\bullet}{B}\left(  t\right)  \right]  dt+A\left(
t\right)  \sigma u_{t}dW\left(  t\right) \\
& +%
{\displaystyle\int_{\Gamma}}
A(t)\left(  1+r\left(  t,\lambda\right)  \right)  u_{t}\widetilde{N}\left(
d\lambda,dt\right)  ,\\
d\widetilde{p}_{2}^{u}\left(  T\right)  = & A\left(  T\right)  x^{u}\left(
T\right)  +B\left(  T\right)  .
\end{array}
\right.  \label{p2 Itoexemple}%
\end{equation}

On the other hand, by substituting $\left(  \ref{conjectionexemple}\right)  $
into $\left(  \ref{p2exemple}\right)  ,$ and denote by%
\begin{equation}%
\begin{array}
[c]{l}%
\widetilde{q}_{3}^{u}\left(  t\right)  =\left(  \theta l_{2}\left(  t\right)
\widetilde{p}_{2}^{u}\left(  t\right)  -\widetilde{q}_{2}^{u}\left(  t\right)
\right) \\
\widetilde{\pi}_{3}\left(  t,.\right)  =\widetilde{\pi}_{2}\left(  t,.\right)
-\theta L_{2}\left(  t,.\right)  \widetilde{p}_{2}^{u}\left(  t\right)  .
\end{array}
\label{q3ident}%
\end{equation}
By using the Girsanov's transformation in $\left(  \ref{p2exemple}\right)  ,$
as in section 2 lemma $\left(  \ref{transform}\right)  $, we obtain%
\begin{equation}
\left\{
\begin{array}
[c]{ll}%
d\widetilde{p}_{2}^{u}\left(  t\right)  = & \left\{  \left(  c+\theta^{2}%
l^{2}\left(  t\right)  +\int_{\Gamma}\theta^{2}L^{2}\left(  t,\lambda\right)
m\left(  d\lambda\right)  \right)  \widetilde{p}_{2}^{u}\left(  t\right)
+c\widetilde{p}_{3}^{u}\left(  t\right)  -\theta l\left(  t\right)
\widetilde{q}_{3}^{u}\left(  t\right)  \right\}  dt\\
& +\widetilde{q}_{3}^{u}\left(  t\right)  dW\left(  t\right)  +%
{\displaystyle\int_{\Gamma}}
\widetilde{\pi}_{3}\left(  t,\lambda\right)  \widetilde{N}\left(
d\lambda,dt\right)  ,\\
\widetilde{p}_{2}^{u}\left(  T\right)  = & 1+\theta\left(  x_{T}%
-y_{0}-a\right)  .
\end{array}
\right.  . \label{p2 conjectionexemple}%
\end{equation}

By equating the coefficients and the final conditions of $\left(
\ref{p2 conjectionexemple}\right)  $ with $\left(  \ref{p2 Itoexemple}\right)
,$ we have%
\begin{equation}%
\begin{array}
[c]{ll}%
\widetilde{\pi}_{3}\left(  t,\lambda\right)  = & A(t)\left(  1+r\left(
t,.\right)  \right)  u_{t},\\
\widetilde{q}_{3}^{u}\left(  t\right)  = & \sigma u_{t}A\left(  t\right)  ,\\
A\left(  T\right)  = & \theta,\\
B\left(  T\right)  = & 1-\theta\left(  y_{0}+a\right)  .
\end{array}
\label{idenexemple}%
\end{equation}

By identifying $\left(  \ref{q3ident}\right)  $ with $\left(
\ref{idenexemple}\right)  ,$ we can rewrite%
\[
\widetilde{q}_{2}^{u}\left(  t\right)  =\theta l_{2}\left(  t\right)  \left(
A\left(  t\right)  x^{u}\left(  t\right)  +B\left(  t\right)  \right)  +\sigma
u_{t}A\left(  t\right)  ,
\]

and
\[
\widetilde{\pi}_{2}^{u}\left(  t,.\right)  =\theta L_{2}\left(  t,.\right)
\left(  A\left(  t\right)  x^{u}\left(  t\right)  +B\left(  t\right)  \right)
+r\left(  t,.\right)  u_{t}A\left(  t\right)  ,
\]
then replacing the both equations $\left(  \ref{idenexemple}\right)  $, and
the last equations of $\widetilde{q}_{2}^{u}\left(  t\right)  $ and
$\widetilde{\pi}_{2}^{u}\left(  t,.\right)  $ into $\left(
\ref{the Hamiltonien derive}\right)  $, we have,%
\[%
\begin{array}
[c]{l}%
\rho\left(  A(t)x^{u}\left(  t\right)  +B(t)\right)  +\rho\widetilde{p}%
_{3}\left(  t\right)  +\sigma\theta l\left(  t\right)  \left(  A(t)x^{u}%
\left(  t\right)  +B(t)\right)  +\sigma^{2}A(t)u_{t}\\
+\int_{\Gamma}\left\{  \left(  1+r\left(  t,\lambda\right)  \right)  \theta
L\left(  t,\lambda\right)  \left(  A(t)x^{u}\left(  t\right)  +B(t)\right)
+\left(  1+r\left(  t,\lambda\right)  \right)  ^{2}A(t)u_{t}-\rho\widetilde
{p}_{3}\left(  t\right)  \right\}  m\left(  d\lambda\right) \\
=0,
\end{array}
\]

then we get,%

\begin{equation}
u\left(  t,x_{t}\right)  =-\frac{\left(  \rho+\sigma\theta l\left(  t\right)
+\int_{\Gamma}\left(  1+r\left(  t,\lambda\right)  \right)  \theta L\left(
t,\lambda\right)  m\left(  d\lambda\right)  \right)  \left(  A(t)x^{u}\left(
t\right)  +B(t)\right)  +\rho\widetilde{p}_{3}\left(  t\right)  }{A(t)G\left(
t\right)  }, \label{u01}%
\end{equation}

where $G\left(  t\right)  =\sigma%
{{}^2}%
-%
{\displaystyle\int_{\Gamma}}
\left(  1+r\left(  t,\lambda\right)  \right)
{{}^2}%
m\left(  d\lambda\right)  .$

In the other side, we have from $\left(  \ref{p2 Itoexemple}\right)  $ and
$\left(  \ref{p2 conjectionexemple}\right)  .$ Then%
\begin{equation}
u_{t}=-\frac{\overset{\bullet}{A}\left(  t\right)  x^{u}\left(  t\right)
-2cA(t)x^{u}\left(  t\right)  -cB(t)+\overset{\bullet}{B}\left(  t\right)
-c\widetilde{p}_{3}^{u}\left(  t\right)  }{A\left(  t\right)  \left(
\rho+\sigma\theta l\left(  t\right)  +\int_{\Gamma}\left(  1+r\left(
t,\lambda\right)  \right)  \theta L\left(  t,\lambda\right)  m\left(
d\lambda\right)  \right)  }. \label{u02}%
\end{equation}

From $\left(  \ref{u01}\right)  $ and $\left(  \ref{u02}\right)  $, we have%
\begin{equation}
\left\{
\begin{array}
[c]{l}%
\overset{\bullet}{A}\left(  t\right)  =\left\{  2c+\dfrac{\left(  \rho
+\sigma\theta l\left(  t\right)  +\int_{\Gamma}\left(  1+r\left(
t,\lambda\right)  \right)  \theta L\left(  t,\lambda\right)  m\left(
d\lambda\right)  \right)
{{}^2}%
}{G\left(  t\right)  }\right\}  A(t),\\
A(T)=\theta.
\end{array}
\right.  \label{ODE A}%
\end{equation}
and
\begin{equation}
\left\{
\begin{array}
[c]{l}%
\overset{\bullet}{B}\left(  t\right)  =\left\{  c+\dfrac{\left(  \rho
+\sigma\theta l\left(  t\right)  +\int_{\Gamma}\left(  1+r\left(
t,\lambda\right)  \right)  \theta L\left(  t,\lambda\right)  m\left(
d\lambda\right)  \right)
{{}^2}%
}{G\left(  t\right)  }\right\}  B(t)+c\widetilde{p}_{3}^{u}\left(  t\right)
,\\
B(T)=1-\theta\left(  y_{0}+a\right)  .
\end{array}
\right.  \label{ODE B}%
\end{equation}

Then the explicit solutions of $\left(  \ref{ODE A}\right)  ,$ and $\left(
\ref{ODE B}\right)  $ have the form%
\begin{equation}
\left\{
\begin{array}
[c]{l}%
A(t)=\theta\exp%
{\displaystyle\int_{t}^{T}}
\left\{  2c+\dfrac{\left(  \rho+\sigma\theta l\left(  s\right)  +\int_{\Gamma
}\left(  1+r\left(  s,\lambda\right)  \right)  \theta L\left(  s,\lambda
\right)  m\left(  d\lambda\right)  \right)
{{}^2}%
}{G\left(  s\right)  }\right\}  ds,\\
B(t)=\left(  1-\theta\left(  y_{0}+a\right)  \right)  \exp%
{\displaystyle\int_{t}^{T}}
\left[  \left\{  c+\dfrac{\left(  \rho+\sigma\theta l\left(  s\right)
+\int_{\Gamma}\left(  1+r\left(  s,\lambda\right)  \right)  \theta L\left(
s,\lambda\right)  m\left(  d\lambda\right)  \right)
{{}^2}%
}{G\left(  s\right)  }\right\}  \right. \\
\left.  B(s)+c\widetilde{p}_{3}^{u}\left(  s\right)  \right]  ds.
\end{array}
\right.  \label{the explicit sol A et B}%
\end{equation}

\begin{remark}
It's very important to remark that the solution of the function $B\left(
t\right)  $ in the form $\left(  \ref{the explicit sol A et B}\right)  $ is
depend to the solution of $\widetilde{p}_{3}\left(  t\right)  .$ If we put
$\widetilde{p}_{3}\left(  t\right)  =\psi\left(  t\right)  y\left(  t\right)
+\varphi\left(  t\right)  ,$ for smooth deterministic functions $\psi,$ and
$\varphi,$ by using the similar technique as an optimal solution in the last
paragraph, to the triplet $\left(  y^{u}\left(  t\right)  ,\widetilde{p}%
_{3}^{u}\left(  t\right)  ,u\left(  t\right)  \right)  $. Then the solutions
of $\psi,$ and $\varphi$ yield respectively the equations%
\begin{equation}
\left\{
\begin{array}
[c]{l}%
\overset{\bullet}{\psi}\left(  t\right)  =\rho%
{{}^2}%
\psi%
{{}^2}%
\left(  t\right)  -\left(  2\lambda\sigma%
{{}^2}%
A(t)-\theta%
{{}^2}%
l%
{{}^2}%
(t)\right)  \psi\left(  t\right)  ,\\
\overset{\bullet}{\varphi}\left(  t\right)  =(\rho\psi\left(  t\right)
+\theta%
{{}^2}%
l%
{{}^2}%
(t)-\lambda)\varphi\left(  t\right)  +K(t),\\
\psi\left(  0\right)  =\theta,\text{ and }\varphi\left(  0\right)
=1-\theta\left(  y_{0}-a\right)  .
\end{array}
\right.  \label{linear of p3  ch4}%
\end{equation}

\end{remark}

The main result in this section, can be given in the form of maximum principle
of mean variance problem with risk sensitive performance.

\begin{theorem}
We assume that the pair $\left(  A\left(  t\right)  ,B(t)\right)  $ has unique
solution given by $\left(  \ref{the explicit sol A et B}\right)  $, the pair
$\left(  \varphi\left(  t\right)  ,\psi(t)\right)  $ has also the explicit
solution of the system $\left(  \ref{linear of p3 ch4}\right)  $. Then the
optimal control of the problem $\left(  \ref{forwardexemple}\right)  $,
$\left(  \ref{backwardexemple}\right)  $ and $\left(  \ref{costexemple}%
\right)  $ has the state feedback form%
\[%
\begin{array}
[c]{l}%
u\left(  t,x_{t},y_{t},r_{t}\left(  .\right)  \right) \\
=-\dfrac{\left(  \rho+\sigma\theta l\left(  t\right)  +\int_{\Gamma}\left(
1+r\left(  t,\lambda\right)  \right)  \theta L\left(  t,\lambda\right)
m\left(  d\lambda\right)  \right)  \left(  A(t)x^{u}\left(  t\right)
+B(t)\right)  +\rho\left(  \psi\left(  t\right)  y^{u}\left(  t\right)
+\varphi\left(  t\right)  \right)  }{A(t)G\left(  t\right)  }.
\end{array}
\]

\end{theorem}

\section{Conclusion and Remarks:}

This paper contains two main results. The first one, Theorem $\ref{NOC}$,
establishes the necessary optimality conditions for the system of fully
coupled FBSDE with risk sensitive performance, using an almost similar scheme
as in Chala \cite{Chala 02, BTT}. The second main result, Theorem $\ref{SOC}$,
suggests sufficient optimality conditions of fully coupled FBSDE given in form
of risk sensitive performance., we note here that our paper is the second
extension of result of Chala \cite{Chala 03} The proof is based on the
convexity conditions of the Hamiltonian function, the initial and terminal
terms of the performance function. It should be noted that the risk sensitive
control problems studied by Lim and Zhou in \cite{Lim-Zhou} are different from
ours. Our results can be compared with maximum principle obtained by Shi and
Wu \cite{Shi-Wu2}, but we have to be able to discuss the generale case -if we
add the jumps diffusion term to our system-. This result it will be discussed
in our next paper. On the other hand, in the case where the system is governed
by mean field type we may take the existing paper established by Djechiche et
al \cite{BTT}. We have generalized this last result into the fully coupled
stochastic differential equation which is motivated by an optimal portfolio
choice problem in financial market specially the model of control cash flow of
a firm or project for example we can setting the model of pricing and managing
an insurance contract, this counterpart without mean field term as in
\cite{BTT}, A problem to be thoroughly addressed in our future paper, where
the system is governed by fully coupled stochastic differential equation of
mean field type, and will be compared with \cite{Ma}. Remarkably, the maximum
principle of risk-neutral obtained by Wu \cite{Wu}, and Yong \cite{Yong} is
quite similar to our theorem $\ref{theoriskneutral}$, but their adjoint
equation and maximum conditions heavily depend on the risk sensitive parameter.

\end{document}